\documentclass[11pt]{amsart}

\usepackage{hyperref}
\hypersetup{bookmarksdepth=3}
  
\usepackage{geometry}
\usepackage{amsmath,amssymb}
\usepackage{
mathrsfs}
\usepackage{enumerate}
\usepackage{esint}

\usepackage{tensor}

\usepackage{amssymb,amsmath,amsthm}

\newtheorem{theorem}{Theorem}
\newtheorem{lemma}[theorem]{Lemma}

\newtheorem{definition}[theorem]{Definition}

\theoremstyle{remark}\newtheorem{remark}[theorem]{Remark}

\newtheorem{proposition}[theorem]{Proposition}

\usepackage{graphicx}

\newcommand{\R}{{\mathbb R}}

\newcommand{\Z}{{\mathbb Z}}

\newcommand{\C}{{\mathbb C}}

\begin{document}

{\let\thefootnote\relax\footnote{Date: 2nd October 2017. 

\textcopyright 2017 by the author. Faithful reproduction of this article, in its entirety, by any means is permitted for noncommercial purposes.}}

\title{NLS in the modulation space $M_{2,q}(\mathbb R)$. }

\author{N. Pattakos}
\address{nikolaos pattakos, department of mathematics, institute for analysis, karlsruhe institute of technology, 76128 karlsruhe, germany }
\email{nikolaos.pattakos@kit.edu}

\begin{abstract}
{We show the existence of weak solutions in the extended sense of the Cauchy problem for the cubic nonlinear Schr\"odinger equation in the modulation space $M_{2,q}^{s}(\mathbb R)$, $1\leq q\leq2$ and $s\geq0.$ In addition, for either $s\geq 0$ and $1\leq q\leq\frac32$ or $\frac32<q\leq 2$ and $s>\frac23-\frac1{q}$ we show that the Cauchy problem is unconditionally wellposed in $M_{2,q}^{s}(\R).$ It is done with the use of the differentiation by parts technique which had been previously used in the periodic setting.}
\end{abstract}

\maketitle
\pagestyle {myheadings}

\begin{section}{introduction and main results}
\markboth{\normalsize N. Pattakos }{\normalsize  NLS in the modulation space $M_{2,q}(\mathbb R)$}

In this paper we study the one dimensional cubic NLS:
\begin{equation}
\label{maineq}
\begin{cases} iu_{t}-u_{xx}\pm|u|^{2}u=0 &,\ (t,x)\in\mathbb R^{2}\\
u(0,x)=u_{0}(x) &,\ x\in\mathbb R\\
\end{cases}
\end{equation}
with initial data $u_{0}$ in the modulation space $M_{2,q}^{s}(\mathbb R).$ We are interested in existence of solutions and in unconditional wellposedness of the problem. Modulation spaces were introduced by Feichtinger in \cite{FEI} and have been used extensively in the study of nonlinear dispersive equations. See \cite{BH} for many of their properties such as embeddings in other known function spaces and equivalent expressions for their norm. Since their introduction, they have become canonical for both time-frequency and phase-space analysis since they provide an excellent substitute in estimates that are known to fail on Lebesgue spaces. 

Let us mention some already known results on local wellposedness of NLS (\ref{maineq}) with initial data in a modulation space. From \cite{FEI} (Proposition $6.9$) it is known that for $s>1/q'$ or $s\geq 0$ and $q=1$ the modulation space $M^{s}_{p,q}(\mathbb R)$ is a Banach algebra and therefore an easy Banach contraction principle argument together with the fact that $e^{it\Delta}$ is a bounded operator from $M_{p,q}^{s}(\R)$ to itself (see \cite{BGOR} and \cite{BO}) implies that NLS (\ref{maineq}) is locally wellposed for $u_{0}\in M^{s}_{p,q}(\mathbb R)$ with solution $u\in C([0,T];M^{s}_{p,q}(\mathbb R))$, $T>0.$ Also in \cite{SG} the case $u_{0}\in M_{2,q}(\mathbb R)$, $2\leq q<\infty$, was considered which is a space that does not belong to the previous family of Banach algebras.

The definition of modulation spaces is the following: Set $Q_{0}=[-\frac12, \frac12)$ and $Q_{k}=Q_{0}+k$ for all $k\in\mathbb Z.$ Consider a family of functions $\{\sigma_{k}\}_{k\in\mathbb Z}\subset C^{\infty}(\mathbb R)$ satisfying 
\begin{itemize}
\item
$ \exists c > 0: \, \forall k \in \Z: \,
\forall \eta \in Q_{k}: \,
|\sigma_{k}(\eta)| \geq c$,
\item
$\forall k \in \Z: \,
\mbox{supp}(\sigma_{k}) \subseteq \{\xi\in\R:|\xi-k|\leq1\}$,
\item
$\sum_{k \in \Z} \sigma_{k} = 1$,
\item
$\forall m \in \mathbb N_0: \, \exists C_m > 0: \,
\forall k \in \Z: \,
\forall \alpha \in \mathbb N: \,
\alpha \leq m \Rightarrow
\|{D^\alpha \sigma_{k}}\|_\infty \leq C_m$
\end{itemize}
and define the isometric decomposition operators
\begin{equation}
\label{iso}
\Box_{k} := \mathcal F^{(-1)} \sigma_{k} \mathcal F, \quad
\left(\forall k \in \Z \right).
\end{equation}
Then the norm of a tempered distribution $f\in S'(\R)$ in the modulation space $M^{s}_{p,q}(\mathbb R)$, $s\in\mathbb R, 1\leq p,q\leq\infty$, is 
\begin{equation}
\label{def}
\|f\|_{M^{s}_{p,q}}:=\Big(\sum_{k\in\mathbb Z}\langle k\rangle^{sq}\|\Box_{k}f\|_{p}^{q}\Big)^{\frac1{q}}=\Big(\sum_{k\in\mathbb Z}(1+|k|^{2})^{\frac{sq}{2}}\|\Box_{k}f\|_{p}^{q}\Big)^{\frac1{q}},
\end{equation}
with the usual interpretation when the index $q$ is equal to infinity. Different choices of such sequences of functions $\{\sigma_{k}\}_{k\in\mathbb Z}$ lead to equivalent norms in $M^{s}_{p,q}(\mathbb R).$ When $s=0$ we denote the space $M^{0}_{p,q}(\mathbb R)$ by $M_{p,q}(\mathbb R).$ In the special case where $p=q=2$ we have $M_{2,2}^{s}(\R)=H^{s}(\R)$ the usual Sobolev spaces. Here we will use that for $s>1/q'$ and $1\leq p, q\leq\infty$, the embedding 
\begin{equation}
\label{yeye}
M_{p,q}^{s}(\R)\hookrightarrow C_{b}(\R)=\{f:\R\to\C/\ f\ \mbox{continuous and bounded}\},
\end{equation}
and for $\Big(1\leq p_{1}\leq p_{2}\leq \infty$, $1\leq q_{1}\leq q_{2}\leq\infty$, $s_{1}\geq s_{2}\Big)$ or $\Big(1\leq p_{1}\leq p_{2}\leq \infty$, $1\leq q_{2}<q_{1}\leq\infty$, $s_{1}>s_{2}+\frac1{q_{2}}-\frac1{q_{1}}\Big)$ the embedding
\begin{equation}
\label{yeye233}
M_{p_{1}, q_{1}}^{s_{1}}(\R)\hookrightarrow M_{p_{2}, q_{2}}^{s_{2}}(\R),
\end{equation}
are both continuous and can be found in \cite{FEI} (Proposition $6.8$ and Proposition $6.5$). Also, by \cite{BH} it is known that for any $1<p\leq\infty$ we have the embedding $M_{p,1}(\R)\hookrightarrow L^{p}(\R)\cap L^{\infty}(\R)$ which together with the fact that $M_{2,2}(\R)=L^{2}(\R)$ and interpolation, imply that for any $p\in[2,\infty]$ we have the embedding $M_{p,p'}(\R)\hookrightarrow L^{p}(\R).$ Later in Subsection $2.4$ we will use this fact for $p=3$, that is 
\begin{equation}
\label{hhh}
M_{3,\frac32}(\R)\hookrightarrow L^{3}(\R).
\end{equation}
In order to give a meaning to solutions of the NLS in $C([0,T], M_{2,q}(\R))$ and to the nonlinearity $\mathcal N(u):=u|u|^{2}$ we need the following definitions which first appeared in \cite{C1}, \cite{C2} where power series solutions to the cubic NLS was studied (see also \cite{GUB} for similar considerations on the KdV):

\begin{definition}
\label{def1}
A sequence of Fourier cutoff operators is a sequence of Fourier multiplier operators $\{T_{N}\}_{N\in \mathbb N}$ on $\mathcal S'(\mathbb R)$ with multipliers $m_{N}:\R\to\C$ such that
\begin{itemize}
\item $m_{N}$ has compact support on $\R$ for every $N\in \mathbb N$,
\item $m_{N}$ is uniformly bounded,
\item $\lim_{N\to\infty}m_{N}(x)=1$, for any $x\in\R$. 
\end{itemize}
\end{definition}

\begin{definition}
\label{def2}
Let $u\in C([0,T],M_{2,q}^{s}(\R))$. We say that $\mathcal N(u)$ exists and is equal to a distribution $w\in\mathcal S'((0,T)\times\R)$ if for every sequence $\{T_{N}\}_{N\in\mathbb N}$ of Fourier cutoff operators we have
\begin{equation}
\label{wknows}
\lim_{N\to\infty}\mathcal N(T_{N}u)=w,
\end{equation}
in the sense of distributions on $(0,T)\times\R$.
\end{definition}

\begin{definition}
\label{def3}
We say that $u\in C([0,T],M_{2,q}^{s}(\R))$ is a weak solution in the extended sense of NLS (\ref{maineq}) if the following are satisfied 
\begin{itemize}
\item $u(0,x)=u_{0}(x)$,
\item the nonlinearity $\mathcal N(u)$ exists in the sense of Definition \ref{def2},
\item $u$ satisfies (\ref{maineq}) in the sense of distributions on $(0,T)\times\R$, where the nonlinearity $\mathcal N(u)=u|u|^{2}$ is interpreted as above.
\end{itemize}
\end{definition}
Our main result which guarantees existence of weak solutions in the extended sense is the following:

\begin{theorem}
\label{th1}
Let $1\leq q\leq2$ and $s\geq0$. For $u_{0}\in M_{2,q}^{s}(\mathbb R)$ there exists a weak solution in the extended sense $u\in C([0,T];M_{2,q}^{s}(\mathbb R))$ of NLS (\ref{maineq}) with initial condition $u_{0}$, where the time $T$ of existence depends only on $\|u_{0}\|_{M_{2,q}^{s}}.$ Moreover, the solution map is Lipschitz continuous. 
\end{theorem}

\begin{remark}
The restriction on the range of $q$ appears by the construction of the solution of the NLS. That is, we decompose the NLS into countably many parts and at the end we sum all of them together. In order for the summation to make sense in the appropriate space we obtain $1\leq q<3$ (see remarks after (\ref{factori}) below). Moreover, when estimating the resonant operator $R_{2}^{t}$ in Lemma \ref{lem} the restriction $q\leq2$ appears naturally. 
\end{remark}
The next theorem is about the unconditional wellposedness of NLS (\ref{maineq}) with initial data in a modulation space, that is, uniqueness in $C([0,T],M_{2,q}^{s}(\R))$ without intersecting with any auxiliary function space (see \cite{TK} where this notion first appeared):

\begin{theorem}
\label{mainyeah}
For $u_{0}\in M_{2,q}^{s}(\R)$, with either $s\geq0$ and $1\leq q\leq\frac32$ or $\frac32<q\leq2$ and $s>\frac23-\frac1{q}$, the solution $u$ with initial condition $u_{0}$ constructed in Theorem \ref{th1} is unique in $C([0,T],M_{2,q}^{s}(\R))$. 
\end{theorem}

\begin{remark}
\label{bbbb}
When $q=2$ the value $s=\frac16$ is also allowed in the previous theorem since then we have the space $M_{2,2}^{\frac16}(\R)=H^{\frac16}(\R)\hookrightarrow L^{3}(\R)$. 
\end{remark}

For its proof we are going to use the differentiation by parts technique that was introduced in \cite{BIT} to attack similar problems for the KdV equation but with periodic initial data. In \cite{GKO} this technique was used to prove unconditional wellposedness of the periodic cubic NLS. In this paper we use this technique to attack an NLS with a continuous Fourier variable, in the sense that our initial data is far from being periodic. For this reason there are some major differences and some difficulties that do not occur in the periodic setting. We follow very closely the ideas of \cite{GKO} but we have to replace numbers and estimates for sums of numbers by operators and estimates for sums of suitable operator norms. This will become clearer in the next section where the proofs of Theorem \ref{th1} and Theorem \ref{mainyeah} will be given. Here let us mention that a similar approach was used in \cite{HY} to study the cubic NLS and the mKdV on the real line and obtain unconditional wellposedness results with initial data in the Sobolev space $H^{s}(\R)$. Finally, similar techniques were used in \cite{JS} to study the quadratic nonlinear Klein-Gordon equation. 

Since we are interested in the space $M_{2,q}^{s}(\R)$ there is a more convenient expression for its norm which is the one we are going to use in our calculations. Let us denote by $\tilde\Box_{k}$ the frequency projection operator $\mathcal F^{(-1)}1_{[k,k+1]}\mathcal F$, where $1_{[k,k+1]}$ is the characteristic function of the interval $[k,k+1]$, $k\in\Z.$ Then it can be proved that 
\begin{equation}
\label{normeq}
\|f\|_{M_{2,q}^{s}}\approx\Big(\sum_{k\in\Z}\langle k\rangle^{sq}\|\tilde\Box_{k}f\|_{2}^{q}\Big)^{\frac1{q}},
\end{equation}
or in other words, the two norms are equivalent in $M_{2,q}^{s}(\R)$.

To conclude this section, firstly, we need that for $S(t)=e^{it\Delta}$ the Schr\"odinger semigroup we have the equality:
\begin{equation}
\label{Sch}
\|S(t)f\|_{2}=\|f\|_{2},
\end{equation}
and secondly, we need the multiplier estimate (see \cite{BH}, Proposition $1.9$):

\begin{lemma}
\label{Bern}
Let $1 \leq p \leq \infty$ and $\sigma \in C^{\infty}_{c}(\R)$. Then the multiplier operator $T_\sigma: S(\R) \to S'(\R)$ defined by
\begin{equation*}
(T_\sigma f) = \mathcal F^{-1}(\sigma \cdot \hat{f}), \quad
\forall f \in S(\R)
\end{equation*}
is bounded on $L^p(\R)$ and
\begin{equation*}
\|T_{\sigma}\|_{L^p(\R)\to L^p(\R)} \lesssim\|\check\sigma\|_{L^{1}(\R)}.
\end{equation*}
\end{lemma}
A useful consequence is that for $1\leq p_{1}\leq p_{2}\leq\infty$ the following holds:
\begin{equation}
\label{Bern1}
\|\Box_{k}f\|_{p_{2}}\lesssim\|\Box_{k}f\|_{p_{1}},
\end{equation}
where the implicit constant is independent of $k$ and the function $f.$ This is done by considering a "fattened" function $\tilde{\sigma}_{0}$ which is identically $1$ on the support of $\sigma_{0}$ and then by defining $\tilde{\sigma}_{k}(\xi)=\tilde{\sigma}_{0}(\xi-k)$, $\tilde\Box_{k}=\mathcal F^{-1}\tilde{\sigma}_{k}\mathcal F$, for $k\in\Z$, we have that  
$$\|\Box_{k}f\|_{p_{2}}=\|\tilde\Box_{k}\Box_{k}f\|_{p_{2}}=\|\mathcal F^{-1}(\tilde{\sigma}_{k})\ast\Box_{k}f\|_{p_{2}}\leq\|\mathcal F^{-1}(\tilde{\sigma}_{k})\|_{r}\|\Box_{k}f\|_{p_{1}}=\|\mathcal F^{-1}(\tilde{\sigma}_{0})\|_{r}\|\Box_{k}f\|_{p_{1}},$$
where we applied Young's inequality with indices $1+\frac1{p_{2}}=\frac1{r}+\frac1{p_{1}}$ and we used that all $\tilde{\sigma}_{k}$ are translations of $\tilde{\sigma}_{0}$.

Let us also recall the following number theoretic fact (see \cite{HW}, Theorem $315$) which is going to be used throughout the proof of Theorem \ref{th1}: Given an integer $m$, let $d(m)$ denote the number of divisors of $m$. Then we have
\begin{equation}
\label{num}
d(m)\lesssim e^{c\frac{\log m}{\log\log m}}=o(m^{\epsilon}),
\end{equation}
for all $\epsilon>0$. 

Lastly, before we proceed into the next section let us fix the notation: For a number $1\leq p\leq\infty$ we write $p'$ for its dual exponent, that is the number that satisfies $\frac1{p}+\frac1{p'}=1.$ We denote by $S'(\R)$ the space of tempered distributions and by $D'(\R)$ the space of distributions. For two quantities $A, B$ (they can be functions or numbers) whenever we write $A\lesssim B$ we mean that there is a universal constant $C>0$ such that $A\leq CB$. 

Next section consists of four subsections. In Subsection $2.1$ the first steps of the iteration process are presented and in Subsection $2.2$ the tree notation and the induction step finish the infinite iteration procedure. Then, in Subsection $2.3$ Theorem \ref{th1} is proved where the solution is constructed through an approximation by smooth solutions and in Subsection $2.4$ the unconditional uniqueness of Theorem \ref{mainyeah} is presented under the extra assumption that the solution lies in the space $C([0,T],L^{3}(\R)).$

\end{section}

\begin{section}{proof of the main theorems}

\subsection{The first steps of the iteration process}

In this subsection we present the first steps of the differentiation by parts technique adapted to the continuous setting, that is NLS (\ref{maineq}) with initial data that is not periodic. Since it is the first time that this is done, we try to be detailed for the interested reader. We will also use the same notation as in \cite{GKO} so that a direct comparison between the two papers can be made and the differences can be emphasised. 

From here on, we consider only the case $s=0$ in Theorem \ref{th1} since for $s>0$ similar considerations apply. See Remark \ref{reme} at the end of Subsection $2.2$ for a more detailed argument. Also, as we mentioned before we are going to use expression (\ref{normeq}) for the norm in $M_{2,q}(\R)$ and for convenience we will write $\Box_{n}$ instead of $\tilde\Box_{n}$ and $\sigma_{k}$ instead of $1_{[k,k+1]}$. 

For $n\in\mathbb Z$ let us define
\begin{equation}
\label{ww0}
u_{n}(t,x)=\Box_{n} u(t,x),
\end{equation}
\begin{equation}
\label{ww1}
v(t,x)=e^{it\partial_{x}^{2}}u(t,x),
\end{equation}
\begin{equation}
\label{ww}
v_{n}(t,x)=e^{it\partial_{x}^{2}}u_{n}(t,x)=\Box_{n}[(e^{it\partial_{x}^{2}}u(t,x)]=\Box_{n}v(t,x).
\end{equation}
Also for $(\xi,\xi_{1},\xi_{2},\xi_{3})\in\mathbb R^{4}$ we define the function
$$\Phi(\xi,\xi_{1},\xi_{2},\xi_{3})=\xi^{2}-\xi_{1}^{2}+\xi_{2}^{2}-\xi_{3}^{2},$$
which is equal to 
$$\Phi(\xi,\xi_{1},\xi_{2},\xi_{3})=2(\xi-\xi_{1})(\xi-\xi_{3}),$$
if $\xi=\xi_{1}-\xi_{2}+\xi_{3}$. Our main equation (\ref{maineq}) implies that
\begin{equation}
\label{main2}
i\partial_{t}u_{n}-(u_{n})_{xx}\pm\Box_{n}(|u|^{2}u)=0,
\end{equation}
and by calculating ($u=\sum_{k}\Box_{k}u$)
$$\Box_{n}(u\bar{u}u)=\Box_{n}\sum_{n_{1},n_{2},n_{3}}u_{n_{1}}\bar{u}_{n_{2}}u_{n_{3}}=\sum_{n_{1}-n_{2}+n_{3}\approx n}\Box_{n}[u_{n_{1}}\bar{u}_{n_{2}}u_{n_{3}}],$$
where by $\approx n$ we mean $=n$ or $=n+1$ or $=n-1.$ Later, during the calculations we will also write $\xi\approx n$ where $\xi$ is going to be a continuous variable and $n$ an integer. By that we will mean that $\xi\in[n,n+1)$ or more generally that $\xi$ is in an interval around $n$.

Next we do the change of variables $u_{n}(t,x)=e^{-it\partial_{x}^{2}}v_{n}(t,x)$ and arrive at the expression
\begin{equation}
\label{main3}
\partial_{t}v_{n}=\pm i\sum_{n_{1}-n_{2}+n_{3}\approx n}\Box_{n}\Big(e^{it\partial_{x}^{2}}[e^{-it\partial_{x}^{2}}v_{n_{1}}\cdot e^{it\partial_{x}^{2}}\bar{v}_{n_{2}}\cdot e^{-it\partial_{x}^{2}}v_{n_{3}}]\Big).
\end{equation}
We continue by presenting the first steps of our splitting procedure. Define the $1$st generation operators by
\begin{equation}
\label{main4}
Q^{1,t}_{n}(v_{n_{1}},\bar{v}_{n_{2}},v_{n_{3}})(x)=\Box_{n}\Big(e^{it\partial_{x}^{2}}[e^{-it\partial_{x}^{2}}v_{n_{1}}\cdot e^{it\partial_{x}^{2}}\bar{v}_{n_{2}}\cdot e^{-it\partial_{x}^{2}}v_{n_{3}}]\Big),
\end{equation}
and continue with the splitting 
\begin{equation}
\label{main11}
\partial_{t}v_{n}=\pm i\sum_{n_{1}-n_{2}+n_{3}\approx n}Q^{1,t}_{n}(v_{n_{1}},\bar{v}_{n_{2}},v_{n_{3}})=\sum_{\substack{n_{1}\approx n\\ or\\ n_{3}\approx n}}\ldots+\sum_{n_{1}\not\approx n\not\approx n_{3}}\ldots.
\end{equation}
We define the resonant part
\begin{equation}
\label{main9}
R^{t}_{2}(v)(n)-R^{t}_{1}(v)(n)=\Big(\sum_{n_{1}\approx n}Q^{1,t}_{n}+\sum_{n_{3}\approx n}Q^{1,t}_{n}\Big)-\sum_{\substack{n_{1}\approx n\\ and\\ n_{3}\approx n}}Q^{1,t}_{n}(v_{n_{1}},\bar{v}_{n_{2}},v_{n_{3}}),
\end{equation}
with $R^{t}_{2}$ being equal to the sum of the first two summands and $R^{t}_{1}$ being equal with the last summand, and the non-resonant part
\begin{equation}
\label{main10}
N_{1}^{t}(v)(n)=\sum_{n_{1}\not\approx n\not\approx n_{3}}Q^{1,t}_{n}(v_{n_{1}},\bar{v}_{n_{2}},v_{n_{3}}),
\end{equation}
which implies the following expression for our NLS (we drop the factor $\pm i$ in front of the sum since they will play no role in our analysis)
\begin{equation}
\label{mainmain}
\partial_{t}v_{n}=R^{t}_{2}(v)(n)-R^{t}_{1}(v)(n)+N_{1}^{t}(v)(n).
\end{equation}

\begin{remark}
In the following part of the paper a series of lemmata will be presented. Unless stated otherwise we will always assume that $1\leq q<\infty$. 
\end{remark}
For the resonant part we have the lemma:

\begin{lemma}
\label{lem}
For $j=1,2$ 
$$\|R^{t}_{j}(v)\|_{l^{q}L^{2}}\lesssim\|v\|^{3}_{M_{2,q}},$$
and
$$\|R^{t}_{j}(v)-R^{t}_{j}(w)\|_{l^{q}L^{2}}\lesssim(\|v\|^{2}_{M_{2,q}}+\|w\|^{2}_{M_{2,q}})\|v-w\|_{M_{2,q}}.$$
\end{lemma}
\begin{proof}
Let us start with $R^{t}_{1}.$ By definition for fixed $n$, $R^{t}_{1}(n)$ consists of finitely many summands, since $|n-n_{1}|, |n-n_{3}|\leq 1$ and $|n-n_{2}|\leq 3$. We will handle $Q^{1,t}_{n}(v_{n},\bar{v}_{n},v_{n})$ and the remaining summands can be treated similarly. Since,
$$Q^{1,t}_{n}(v_{n},\bar{v}_{n},v_{n})=\Box_{n}\Big(e^{it\partial_{x}^{2}}[e^{-it\partial_{x}^{2}}v_{n}\cdot e^{it\partial_{x}^{2}}\bar{v}_{n}\cdot e^{-it\partial_{x}^{2}}v_{n}]\Big),$$
and since the Schr\"odinger operator is an isometry on $L^{2}$ our claim follows by Bernstein's inequality (see Lemma \ref{Bern}). For the difference  $R_{1}^{t}(v)-R_{1}^{t}(w)$ we have to estimate terms of the form $|e^{-it\partial_{x}^{2}}v_{n}|^{2}|e^{-it\partial_{x}^{2}}v_{n}-e^{-it\partial_{x}^{2}}w_{n}|$ in the $l^{q}L^{2}$ norm. For the $L^{2}$ norm we apply H\"older's inequality and obtain the upper bound
$$\|e^{-it\partial_{x}^{2}}v_{n}\|_{8}^{2}\|e^{-it\partial_{x}^{2}}v_{n}-e^{-it\partial_{x}^{2}}w_{n}\|_{4}\lesssim\|e^{-it\partial_{x}^{2}}v_{n}\|_{2}^{2}\|e^{-it\partial_{x}^{2}}v_{n}-e^{-it\partial_{x}^{2}}w_{n}\|_{2}=\|v_{n}\|_{2}^{2}\|v_{n}-w_{n}\|_{2},$$
where we used (\ref{Bern1}) and (\ref{Sch}), and then proceed with the $l^{q}$ norm as
$$\Big(\sum_{n\in\Z}\|v_{n}\|_{2}^{2q}\|v_{n}-w_{n}\|_{2}^{q}\Big)^{\frac1{q}}\leq\Big(\sup_{n\in\Z}\|v_{n}\|_{2}^{2}\Big)\Big(\sum_{n\in\Z}\|v_{n}-w_{n}\|_{2}^{q}\Big)^{\frac1{q}}=\|v\|_{M_{2,\infty}}^{2}\|v-w\|_{M_{2,q}}.$$
From (\ref{yeye233}) we have $\|v\|_{M_{2,\infty}}\leq\|v\|_{M_{2,q}}$ which finishes the proof. Similar considerations apply to all other lemmata of the paper where estimates of the same form appear. 

For the $R_{2}^{t}$ operator, it suffices to estimate the sum
$$\sum_{\substack{n_{1}-n_{2}+n_{3}\approx n\\ n_{1}\approx n}}Q^{t}_{n}(v_{n_{1}},\bar{v}_{n_{2}},v_{n_{3}})$$
which consists of finitely many sums depending on whether $n_{1}=n-1$, or $n_{1}=n$, or $n_{1}=n+1.$ Let us only treat 
$$\Box_{n}\ e^{it\partial_{x}^{2}}\ \Big(e^{-it\partial_{x}^{2}}v_{n}\sum_{n_{2}\in\Z}|e^{-it\partial_{x}^{2}}v_{n_{2}}|^{2}\Big),$$
since for the remaining sums similar considerations apply. The $L^{2}$ norm equals
$$\Big\|\Box_{n} \Big(u_{n}\sum_{n_{2}\in\Z}|u_{n_{2}}|^{2}\Big)\Big\|_{2}\lesssim\Big\|u_{n}\sum_{n_{2}\in\Z}|u_{n_{2}}|^{2}\Big\|_{2}\leq\sum_{n_{2}\in\Z}\Big\|u_{n}|u_{n_{2}}|^{2}\Big\|_{2}\leq\sum_{n_{2}\in\Z}\|u_{n}\|_{4}\|u_{n_{2}}\|_{8}^{2},$$
where we used that the Schr\"odinger operator is an isometry in $L^{2}$, Lemma \ref{Bern} and H\"older's inequality. With the use of (\ref{Bern1}) this last sum is bounded from above by $\|u_{n}\|_{2}\|u\|_{M_{2,2}}^{2}$ and since $1\leq q\leq2$ we can use the embedding $l^{q}\hookrightarrow l^{2}$ to arrive at $\|u_{n}\|_{2}\|u\|_{M_{2,q}}^{2}.$ Then, the $l^{q}$ norm in the discrete variable implies
$$\|R_{2}^{t}(v)\|_{l^{q}L^{2}}\lesssim\|v\|_{M_{2,q}}^{3}.$$
\end{proof}
For the non-resonant part $N_{1}^{t}$ we have to split as
\begin{equation}
\label{main13}
N_{1}^{t}(v)(n)=N_{11}^{t}(v)(n)+N_{12}^{t}(v)(n),
\end{equation}
where 
$$N_{11}^{t}(v)(n)=\sum_{A_{N}(n)}Q^{1,t}_{n}(v_{n_{1}},\bar{v}_{n_{2}},v_{n_{3}}),$$
and 
\begin{equation}
\label{set1}
A_{N}(n)=\{(n_{1},n_{2},n_{3})\in\mathbb Z^3:n_{1}-n_{2}+n_{3}\approx n, n_{1}\not\approx n\not\approx n_{3}, |\Phi(n,n_{1},n_{2},n_{3})|\leq N\}.
\end{equation}
We also define the set
\begin{equation}
\label{idid}
A_{N}(n)^{c}=\{(n_{1},n_{2},n_{3})\in\mathbb Z^3:n_{1}-n_{2}+n_{3}\approx n, n_{1}\not\approx n\not\approx n_{3}, |\Phi(n,n_{1},n_{2},n_{3})|> N\}.
\end{equation}
The number $N>0$ is considered to be large and will be fixed at the end of the proof. With the use of inequality (\ref{num}) we estimate $N_{11}^{t}$ as follows:

\begin{lemma}
\label{lemle}
$$\|N_{11}^{t}(v)\|_{l^{q}L^{2}}\lesssim N^{\frac1{q'}+}\|v\|^{3}_{M_{2,q}},$$
and
$$\|N_{11}^{t}(v)-N_{11}^{t}(w)\|_{l^{q}L^{2}}\lesssim N^{\frac1{q'}+}(\|v\|^{2}_{M_{2,q}}+\|w\|^{2}_{M_{2,q}})\|v-w\|_{M_{2,q}}.$$
\end{lemma}
\begin{proof}
Obviously, 
$$\|N_{11}^{t}(v)\|_{L^{2}}\leq\sum_{A_{N}(n)}\|Q^{1,t}_{n}(v_{n_{1}},\bar{v}_{n_{2}},v_{n_{3}})\|_{L^{2}},$$
which from (\ref{Sch}), Lemma \ref{Bern} and H\"older's inequality is estimated above by
$$\sum_{A_{N}(n)}\|u_{n_{1}}\bar{u}_{n_{2}}u_{n_{3}}\|_{L^{2}}\leq\sum_{A_{N}(n)}\|u_{n_{1}}\|_{L^{6}}\|u_{n_{2}}\|_{L^{6}}\|u_{n_{3}}\|_{L^{6}}.$$
Here we make use of (\ref{Bern1}) and H\"older's inequality in the discrete variable to obtain the upper bound
$$\sum_{A_{N}(n)}\|u_{n_{1}}\|_{L^{2}}\|u_{n_{2}}\|_{L^{2}}\|u_{n_{3}}\|_{L^{2}}\leq\Big(\sum_{A_{N}(n)}1^{q'}\Big)^{\frac1{q'}}\Big(\sum_{A_{N}(n)}\|u_{n_{1}}\|_{L^{2}}^{q}\|u_{n_{2}}\|_{L^{2}}^{q}\|u_{n_{3}}\|_{L^{2}}^{q}\Big)^{\frac1{q}}.$$
Fix $n$ and $\mu\in\mathbb Z$ such that $|\mu|\leq N$. From (\ref{num}) there are at most $o(N^{0+})$ many choices for $n_{1}$ and $n_{3}$, and so for $n_{2}$ from $n\approx n_{1}-n_{2}+n_{3}$, satisfying 
$$\mu=2(n-n_{1})(n-n_{3}).$$
Therefore, we arrive at
$$\|N_{11}^{t}(v)\|_{l^{q}L^{2}}\lesssim N^{\frac1{q'}+}\Big(\sum_{n\in\mathbb Z}\sum_{A_{N}(n)}\|u_{n_{1}}\|_{L^{2}}^{q}\|u_{n_{2}}\|_{L^{2}}^{q}\|u_{n_{3}}\|_{L^{2}}^{q}\Big)^{\frac1{q}},$$
and this final summation is estimated by Young's inequality providing us with the bound ($\|u\|_{M_{2,q}}=\|v\|_{M_{2,q}}$)
$$\|N_{11}^{t}(v)\|_{l^{q}L^{2}}\lesssim N^{\frac1{q'}+}\|v\|_{M_{2,q}}^{3}.$$
\end{proof}
In order to continue, we have to look at the $N_{12}^{t}$ part more closely keeping in mind that we are on $A_{N}(n)^{c}$. Our goal is to find a suitable splitting in order to continue our iteration. In the following we perform all formal calculations assuming that $v$ is a sufficiently smooth solution. Later, in Subsection $2.4$ we justify these formal computations also for $v\in C([0,T], M_{2,q}^{s}(\R))$, with $1\leq q\leq\frac32$, $s\geq 0$ or $\frac32<q\leq2$, $s>\frac23-\frac1{q}$. 

From (\ref{main4}) we know that 
$$\mathcal F(Q^{1,t}_{n}(v_{n_{1}},\bar{v}_{n_{2}},v_{n_{3}}))(\xi)=\sigma_{n}(\xi)\int_{\mathbb R^2}e^{-2it(\xi-\xi_{1})(\xi-\xi_{3})}\hat{v}_{n_{1}}(\xi_{1})\hat{\bar{v}}_{n_{2}}(\xi-\xi_{1}-\xi_{3})\hat{v}_{n_{3}}(\xi_{3})\ d\xi_{1}d\xi_{3},$$
and by the usual product rule for the derivative we can write the previous integral as the sum of the following expressions 
\begin{equation}
\label{ttr}
\partial_{t}\Big(\sigma_{n}(\xi)\int_{\mathbb R^2}\frac{e^{-2it(\xi-\xi_{1})(\xi-\xi_{3})}}{-2i(\xi-\xi_{1})(\xi-\xi_{3})}\ \hat{v}_{n_{1}}(\xi_{1})\hat{\bar{v}}_{n_{2}}(\xi-\xi_{1}-\xi_{3})\hat{v}_{n_{3}}(\xi_{3})\ d\xi_{1}d\xi_{3}\Big)-
\end{equation}
$$\sigma_{n}(\xi)\int_{\mathbb R^2}\frac{e^{-2it(\xi-\xi_{1})(\xi-\xi_{3})}}{-2i(\xi-\xi_{1})(\xi-\xi_{3})}\partial_{t}\Big(\hat{v}_{n_{1}}(\xi_{1})\hat{\bar{v}}_{n_{2}}(\xi-\xi_{1}-\xi_{3})\hat{v}_{n_{3}}(\xi_{3})\Big)\ d\xi_{1}d\xi_{3}.$$ 
Therefore, we have the splitting 
\begin{equation}
\label{main5}
\mathcal F(Q^{1,t}_{n})=\partial_{t}\mathcal F(\tilde{Q}^{1,t}_{n})-\mathcal F(T^{1,t}_{n})
\end{equation}
or equivalently
\begin{equation}
\label{main6}
Q^{1,t}_{n}(v_{n_{1}},\bar{v}_{n_{2}},v_{n_{3}})=\partial_{t}(\tilde{Q}^{1,t}_{n}(v_{n_{1}},\bar{v}_{n_{2}},v_{n_{3}}))-T^{1,t}_{n}(v_{n_{1}},\bar{v}_{n_{2}},v_{n_{3}}),
\end{equation}
which allows us to write 
\begin{equation}
\label{nex}
N_{12}^{t}(v)(n)=\partial_{t}(N_{21}^{t}(v)(n))+N_{22}^{t}(v)(n),
\end{equation}
where
\begin{equation}
\label{nex1}
N_{21}^{t}(v)(n)=\sum_{A_{N}(n)^{c}}\tilde{Q}^{1,t}_{n}(v_{n_{1}},\bar{v}_{n_{2}},v_{n_{3}}),
\end{equation}
and
\begin{equation}
\label{nex2}
N_{22}^{t}(v)(n)=\sum_{A_{N}(n)^{c}}T_{n}^{1,t}(v_{n_{1}},\bar{v}_{n_{2}},v_{n_{3}}).
\end{equation}
Moreover, we have 
$$\mathcal F(\tilde{Q}^{1,t}_{n}(v_{n_{1}},\bar{v}_{n_{2}},v_{n_{3}}))(\xi)=e^{-it\xi^{2}}\sigma_{n}(\xi)\int_{\mathbb R^2}\frac{\hat{u}_{n_{1}}(\xi_{1})\hat{\bar{u}}_{n_{2}}(\xi-\xi_{1}-\xi_{3})\hat{u}_{n_{3}}(\xi_{3})}{(\xi-\xi_{1})(\xi-\xi_{3})}\ d\xi_{1}d\xi_{3},$$
and we define 
\begin{equation}
\label{rr}
\mathcal F(R^{1,t}_{n}(u_{n_{1}},\bar{u}_{n_{2}},u_{n_{3}}))(\xi)=\sigma_{n}(\xi)\int_{\mathbb R^2}\frac{\hat{u}_{n_{1}}(\xi_{1})\hat{\bar{u}}_{n_{2}}(\xi-\xi_{1}-\xi_{3})\hat{u}_{n_{3}}(\xi_{3})}{(\xi-\xi_{1})(\xi-\xi_{3})}\ d\xi_{1}d\xi_{3},
\end{equation}
which is the same as the operator
\begin{equation}
\label{rr1}
R^{1,t}_{n}(w_{n_{1}},\bar{w}_{n_{2}},w_{n_{3}})(x)=\int_{\mathbb R^3}e^{ix\xi}\sigma_{n}(\xi)\frac{\hat{w}_{n_{1}}(\xi_{1})\hat{\bar{w}}_{n_{2}}(\xi-\xi_{1}-\xi_{3})\hat{w}_{n_{3}}(\xi_{3})}{(\xi-\xi_{1})(\xi-\xi_{3})}\ d\xi_{1}d\xi_{3}d\xi.
\end{equation}
Writing out the Fourier transforms of the functions inside the integral it is not difficult to see that
\begin{equation}
\label{main7}
R^{1,t}_{n}(w_{n_{1}},\bar{w}_{n_{2}},w_{n_{3}})(x)=\int_{\mathbb R^3}K^{(1)}_{n}(x,x_{1},y,x_{3})w_{n_{1}}(x)\bar{w}_{n_{2}}(y)w_{n_{3}}(x_{3})\ dx_{1}dydx_{3},
\end{equation}
where
$$K^{(1)}_{n}(x,x_{1},y,x_{3})=\int_{\mathbb R^3}e^{i\xi_{1}(x-x_{1})+i\eta(x-y)+i\xi_{3}(x-x_{3})}\ \frac{\sigma_{n}(\xi_{1}+\eta+\xi_{3})}{(\eta+\xi_{1})(\eta+\xi_{3})}\ d\xi_{1}d\eta d\xi_{3}=$$
$$\mathcal F^{-1}\rho^{(1)}_{n}(x-x_{1},x-y,x-x_{3})$$
and 
$$\rho_{n}^{(1)}(\xi_{1},\eta,\xi_{3})=\frac{\sigma_{n}(\xi_{1}+\eta+\xi_{3})}{(\eta+\xi_{1})(\eta+\xi_{3})}.$$
The important estimate that the operator $\tilde Q^{1,t}_{n}$ satisfies is described in:

\begin{lemma}
\label{fir}
\begin{equation}
\|\tilde{Q}^{1,t}_{n}(v_{n_{1}},\bar{v}_{n_{2}},v_{n_{3}})\|_{2}\lesssim\frac{\|v_{n_{1}}\|_{2}\|v_{n_{2}}\|_{2}\|v_{n_{3}}\|_{2}}{|n-n_{1}||n-n_{3}|}.
\end{equation}
\end{lemma}
\begin{proof}
Observing that $\mathcal F(\tilde{Q}^{1,t}_{n}(v_{n_{1}},\bar{v}_{n_{2}},v_{n_{3}}))(\xi)=e^{-it\xi^{2}}\mathcal F(R^{1,t}_{n}(v_{n_{1}},\bar{v}_{n_{2}},v_{n_{3}}))(\xi)$ it suffices to estimate the $L^{2}$ norm of the operator $R^{1,t}_{n}$. By duality, let $g\in L^2$, $\|g\|_{2}\neq0$, and consider the pairing
\begin{equation}
\label{thg}
|\langle R^{1,t}_{n}(v_{n_{1}},\bar{v}_{n_{2}},v_{n_{3}}), g\rangle|=\Big|\int_{\R}\mathcal F(R^{1,t}_{n}(v_{n_{1}},\bar{v}_{n_{2}},v_{n_{3}}))(\xi)\mathcal F(g)(\xi)\ d\xi\Big|=\end{equation}
$$\Big|\int_{\R^3}\hat{g}(\xi)\ \sigma_{n}(\xi)\ \frac{\hat{v}_{n_{1}}(\xi_{1})\hat{\bar{v}}_{n_{2}}(\xi-\xi_{1}-\xi_{3})\hat{v}_{n_{3}}(\xi_{3})}{(\xi-\xi_{1})(\xi-\xi_{3})}\ d\xi d\xi_{1} d\xi_{3}\Big|=$$
$$\Big|\int_{\R^3}\hat{g}(\xi_{1}+\eta+\xi_{3})\ \frac{\sigma_{n}(\xi_{1}+\eta+\xi_{3})}{(\eta+\xi_{1})(\eta+\xi_{3})}\ \hat{v}_{n_{1}}(\xi_{1})\hat{\bar{v}}_{n_{2}}(\eta)\hat{v}_{n_{3}}(\xi_{3})\ d\eta d\xi_{1} d\xi_{3}\Big|=$$
$$\Big|\int_{I_{n_{1}}}\int_{I_{n_{2}}}\int_{I_{n_{3}}}\hat{g}(\xi_{1}+\eta+\xi_{3})\ \rho^{(1)}_{n}(\xi_{1},\eta,\xi_{3})\ \hat{v}_{n_{1}}(\xi_{1})\hat{\bar{v}}_{n_{2}}(\eta)\hat{v}_{n_{3}}(\xi_{3})\ d\xi_{1} d\eta d\xi_{3}\Big|,$$
where these three intervals are the compact supports of the functions $\hat{v}_{n_{1}},\hat{\bar{v}}_{n_{2}},\hat{v}_{n_{3}}$ (see (\ref{ww})). By H\"older's inequality we obtain the upper bound
$$\|\rho^{(1)}_{n}\|_{\infty}\|v_{n_{1}}\|_{2}\|v_{n_{2}}\|_{2}\|v_{n_{3}}\|_{2}\Big(\int_{I_{n_{1}}}\int_{I_{n_{2}}}\int_{I_{n_{3}}}|\hat{g}(\xi_{1}+\eta+\xi_{3})|^{2}\ d\xi_{1} d\eta d\xi_{3}\Big)^{\frac12},$$
and the last triple integral is easily estimated by
$$\|\hat{g}\|_{2}\ (|I_{n_{2}}||I_{n_{3}}|)^{\frac12}=\|g\|_{2}\ (|I_{n_{2}}||I_{n_{3}}|)^{\frac12}.$$
Therefore, the following is true
$$\|\tilde{Q}^{1,t}_{n}(v_{n_{1}},\bar{v}_{n_{2}},v_{n_{3}}))\|_{2}=\|R^{1,t}_{n}(v_{n_{1}},\bar{v}_{n_{2}},v_{n_{3}}))\|_{2}\lesssim\|\rho_{n}^{(1)}\|_{\infty}\|v_{n_{1}}\|_{2}\|v_{n_{2}}\|_{2}\|v_{n_{3}}\|_{2},$$
and since $\xi_{1}\approx n_{1}$, $\eta\approx-n_{2}$ and $\xi_{3}\approx n_{3}$ we obtain
$$\|\rho_{n}^{(1)}\|_{\infty}\lesssim\frac1{|n-n_{1}||n-n_{3}|},$$
which finishes the proof.
\end{proof}

\begin{remark}
\label{expl}
Notice that Lemma \ref{fir} (this observation applies to Lemma \ref{indu} too) is true for any triple of functions $f,g,h$ that lie in $M_{2,q}(\R)$ and the only important property is that they are nicely localised on the Fourier side since we consider their box operators $\Box_{n_{1}}f, \Box_{n_{2}}g$ and $\Box_{n_{3}}h.$ Also, the same proof implies that the operator $Q^{1,t}_{n}(v_{n_{1}},\bar{v}_{n_{2}},v_{n_{3}})$ satisfies the estimate
\begin{equation}
\label{imppo}
\|Q^{1,t}_{n}(v_{n_{1}},\bar{v}_{n_{2}},v_{n_{3}})\|_{2}\lesssim\|v_{n_{1}}\|_{2}\|v_{n_{2}}\|_{2}\|v_{n_{3}}\|_{2}.
\end{equation}
These observations will play an important role in Lemma \ref{dankda} of Subsection $2.3$ and Lemma \ref{finafinafina} of Subsection $2.4$.
\end{remark}
Here is the estimate for the $N_{21}^{t}$ operator:

\begin{lemma}
\label{fir1}
$$\|N_{21}^{t}(v)\|_{l^{q}L^{2}}\lesssim N^{\frac1{q'}-1+}\|v\|^{3}_{M_{2,q}},$$
and
$$\|N_{21}^{t}(v)-N_{21}^{t}(w)\|_{l^{q}L^{2}}\lesssim N^{\frac1{q'}-1+}(\|v\|_{M_{2,q}}^{2}+\|w\|_{M_{2,q}}^{2})\|v-w\|_{M_{2,q}}.$$
\end{lemma}
\begin{proof}
From Lemma \ref{fir} we have
$$\|N_{21}^{t}(v)\|_{2}\leq\sum_{A_{N}(n)^{c}}\|\tilde{Q}^{1,t}_{n}(v_{n_{1}},\bar{v}_{n_{2}},v_{n_{3}})\|_{2}\lesssim\sum_{A_{N}(n)^{c}}\frac{\|v_{n_{1}}\|_{2}\|v_{n_{2}}\|_{2}\|v_{n_{3}}\|_{2}}{|n-n_{1}||n-n_{3}|},$$
and by H\"older's inequality the upper bound
$$\Big(\sum_{A_{N}(n)^{c}}\frac1{(|n-n_{1}||n-n_{3}|)^{q'}}\Big)^{\frac1{q'}}\Big(\sum_{A_{N}(n)^{c}}\|v_{n_{1}}\|_{2}^{q}\|v_{n_{2}}\|_{2}^{q}\|v_{n_{3}}\|_{2}^{q}\Big)^{\frac1{q}}.$$
The first sum (for $\mu=|n-n_{1}||n-n_{3}|$) is estimated from above by (with the use of (\ref{num}))
$$\Big(\sum_{\mu=N+1}^{\infty}\frac{\mu^{\epsilon}}{\mu^{q'}}\Big)^{\frac1{q'}}\sim (N^{\epsilon+1-q'})^{\frac1{q'}}=N^{\frac1{q'}-1+},$$
and then with the use of Young's inequality we arrive at
$$\|N_{21}^{t}(v)\|_{l^{q}L^{2}}\lesssim N^{\frac1{q'}-1+}\|v\|^{3}_{M_{2,q}}$$
as claimed. 
\end{proof}
To the remaining part $N_{22}^{t}$ we have to make use of equality (\ref{mainmain}) depending on whether the derivative falls on $\hat{v}_{n_{1}}$ or $\hat{\bar{v}}_{n_{2}}$ or $\hat{v}_{n_{3}}$. Let us see how we can proceed from here:
$$N_{22}^{t}(v)(n)=-2i\sum_{A_{N}(n)^{c}}\Big[\tilde{Q}^{1,t}_{n}(R^{t}_{2}(v)(n_{1})-R^{t}_{1}(v)(n_{1}),\bar{v}_{n_{2}},v_{n_{3}})+\tilde{Q}^{1,t}_{n}(N_{1}^{t}(v)(n_{1}),\bar{v}_{n_{2}},v_{n_{3}})\Big]$$
plus the corresponding term for $\partial_{t}\hat{\bar{v}}_{n_{2}}$ (the number $2$ that appears in front of the previous sum is because the expression is symmetric with respect to $v_{n_{1}}$ and $v_{n_{3}}$). Therefore, we can write $N_{22}^{t}$ as a sum
\begin{equation}
\label{patel2}
N_{22}^{t}(v)(n)=N_{4}^{t}(v)(n)+N_{3}^{t}(v)(n),
\end{equation}
where $N_{4}^{t}(v)(n)$ is the sum with the resonant part $R^{t}_{2}-R^{t}_{1}.$ The following Lemma is true:

\begin{lemma}
\label{fir2}
$$\|N_{4}^{t}(v)\|_{l^{q}L^2}\lesssim N^{\frac1{q'}-1+}\|v\|_{M_{2,q}}^{5},$$
and
$$\|N_{4}^{t}(v)-N_{4}^{t}(w)\|_{l^{q}L^2}\lesssim N^{\frac1{q'}-1+}(\|v\|_{M_{2,q}}^{4}+\|w\|_{M_{2,q}}^{4})\|v-w\|_{M_{2,q}}.$$
\end{lemma}
\begin{proof}
Follows by Lemmata \ref{lem} and \ref{fir1} in the sense that we repeat the proof of Lemma \ref{fir1} and apply Lemma \ref{lem} to the part $R_{2}^{t}(v)(n_{1})-R_{1}^{t}(v)(n_{1})$.
\end{proof}
To continue, we have to decompose $N_{3}^{t}$ even further. It consists of $3$ sums depending on where the operator $N_{1}^{t}$ acts. One of them is the following (similar considerations apply for the remaining sums too)
\begin{equation}
\label{newnew1}
\sum_{A_{N}(n)^{c}}\tilde{Q}^{1,t}_{n}(N_{1}^{t}(v)(n_{1}),\bar{v}_{n_{2}},v_{n_{3}}),
\end{equation}
where
$$N_{1}^{t}(v)(n_{1})=\sum_{m_{1}\not\approx n_{1}\not\approx m_{3}}Q^{1,t}_{n_{1}}(v_{m_{1}},\bar{v}_{m_{2}},v_{m_{3}}),$$
and $n_{1}\approx m_{1}-m_{2}+m_{3}$. Here we have to consider new restrictions on the frequencies $(m_{1},m_{2},m_{3},n_{2},n_{3})$ where the "new" triple of frequencies $m_{1},m_{2},m_{3}$ appears as a "child" of the frequency $n_{1}$. Thus, we define the set ($\mu_{1}=\Phi(n,n_{1},n_{2},n_{3})$ and $\mu_{2}=\Phi(n_{1},m_{1},m_{2},m_{3})$)
\begin{equation}
\label{setset1}
C_{1}=\{|\mu_{1}+\mu_{2}|\leq 5^{3}|\mu_{1}|^{1-\frac1{100}}\},
\end{equation}
and split the sum in (\ref{newnew1}) as 
\begin{equation}
\label{patel}
\sum_{A_{N}(n)^{c}}\sum_{C_{1}}\ldots+\sum_{A_{N}(n)^{c}}\sum_{C_{1}^{c}}\ldots=N_{31}^{t}(v)(n)+N_{32}^{t}(v)(n).
\end{equation}
The following holds:

\begin{lemma}
\label{fir3}
$$\|N_{31}^{t}(v)\|_{l^{q}L^{2}}\lesssim N^{\frac2{q'}-\frac1{100q'}-1+}\|v\|_{M_{2,q}}^{5},$$
and
$$\|N_{31}^{t}(v)-N_{31}^{t}(w)\|_{l^{q}L^{2}}\lesssim N^{\frac2{q'}-\frac1{100q'}-1+}(\|v\|^{4}_{M_{2,q}}+\|w\|_{M_{2,q}}^{4})\|v-w\|_{M_{2,q}}.$$
\end{lemma}
\begin{proof}
From (\ref{num}) we know that for fixed $n$ and $\mu_{1}$, there are at most $o(|\mu_{1}|^{+})$ many choices for $n_{1}$ and $n_{3}$ and for fixed $n_{1}$ and $\mu_{2}$ there are at most $o(|\mu_{2}^{+})$ many choices for $m_{1}$ and $m_{3}$. From (\ref{setset1}) we can control $\mu_{2}$ in terms of $\mu_{1}$, that is $|\mu_{2}|\sim|\mu_{1}|$. In addition, for fixed $|\mu_{1}|$ there are at most $O(|\mu_{1}|^{1-\frac1{100}})$ many choices for $\mu_{2}.$ Therefore,
$$\|N_{31}^{t}(v)\|_{2}\leq\sum_{A_{N}(n)^{c}}\sum_{C_{1}}\|\tilde{Q}^{1,t}_{n}(Q^{1,t}_{n_{1}}(v_{m_{1}},\bar{v}_{m_{2}},v_{m_{3}}),\bar{v}_{n_{2}},v_{n_{3}})\|_{2}\lesssim$$
$$\sum_{A_{N}(n)^{c}}\sum_{C_{1}}\frac{\|v_{m_{1}}\|_{2}\|v_{m_{2}}\|_{2}\|v_{m_{3}}\|_{2}\|v_{n_{2}}\|_{2}\|v_{n_{3}}\|_{2}}{|n-n_{1}||n-n_{3}|}\leq$$
$$\Big(\sum_{\mu=N+1}^{\infty}\frac{\mu^{1-\frac1{100}+}}{\mu^{q'}}\Big)^{\frac1{q'}}\Big(\sum_{A_{N}(n)^{c}}\sum_{C_{1}}\|v_{m_{1}}\|_{2}^{q}\|v_{m_{2}}\|_{2}^{q}\|v_{m_{3}}\|_{2}^{q}\|v_{n_{2}}\|_{2}^{q}\|v_{n_{3}}\|_{2}^{q}\Big)^{\frac1{q}},$$
and then by taking the $l^{q}$ norm in $n$ and applying Young's inequality we are led to the desired estimate. 
\end{proof}
For the $N_{32}^{t}$ part we have to do the differentiation by parts technique which will create the $2$nd generation operators. Our first $2$nd generation operator $Q^{2,t}_{n}$ consists of three sums 
$$q^{2,t}_{1,n}=\sum_{A_{N}(n)^{c}}\sum_{ C_{1}^{c}}\tilde{Q}^{1,t}_{n}(N_{1}^{t}(v)(n_{1}),\bar{v}_{n_{2}},v_{n_{3}}),$$
$$q^{2,t}_{2,n}=\sum_{A_{N}(n)^{c}}\sum_{ C_{1}^{c}}\tilde{Q}^{1,t}_{n}(v_{n_{1}},\overline{N_{1}^{t}(v)}(n_{2}),v_{n_{3}}),$$
$$q^{2,t}_{3,n}=\sum_{A_{N}(n)^{c}}\sum_{ C_{1}^{c}}\tilde{Q}^{1,t}_{n}(v_{n_{1}},\bar{v}_{n_{2}},N_{1}^{t}(v)(n_{3})).$$
Let us have a look at the first sum $q^{2,t}_{1,n}$ (we treat the other two in a similar manner). Its Fourier transform is equal to 
$$\sum_{A_{N}(n)^{c}}\sum_{ C_{1}^{c}}\sigma_{n}(\xi)\int_{\mathbb R^2}\frac{e^{-2it(\xi-\xi_{1})(\xi-\xi_{3})}}{(\xi-\xi_{1})(\xi-\xi_{3})}\ \mathcal F(N_{1}^{t}(v)(n_{1}))(\xi_{1})\hat{\bar{v}}_{n_{2}}(\xi-\xi_{1}-\xi_{3})\hat{v}_{n_{3}}(\xi_{3})\ d\xi_{1}d\xi_{3},$$
where
$$\mathcal F(N_{1}^{t}(v)(n_{1}))(\xi_{1})$$
equals
$$\sum_{\substack {n_{1}\approx m_{1}-m_{2}+m_{3} \\ m_{1}\not\approx n_{1}\not\approx m_{3}}}\sigma_{n_{1}}(\xi_{1})\int_{\mathbb R^2}e^{-2it(\xi_{1}-\xi_{1}')(\xi_{1}-\xi_{3}')}\hat{v}_{m_{1}}(\xi_{1}')\hat{\bar{v}}_{m_{2}}(\xi_{1}-\xi_{1}'-\xi_{3}')\hat{v}_{m_{3}}(\xi_{3}')\ d\xi_{1}'d\xi_{3}'.$$
Putting everything together and applying differentiation by parts we can write the integrals inside the sums as
$$\partial_{t}\Big(\sigma_{n}(\xi)\int_{\mathbb R^4}\sigma_{n_{1}}(\xi_{1})\frac{e^{-it(\mu_{1}+\mu_{2})}}{\mu_{1}(\mu_{1}+\mu_{2})}\hat{v}_{m_{1}}(\xi_{1}')\hat{\bar{v}}_{m_{2}}(\xi_{1}-\xi_{1}'-\xi_{3}')\hat{v}_{m_{3}}(\xi_{3}')\hat{\bar{v}}_{n_{2}}(\xi-\xi_{1}-\xi_{3})\hat{v}_{n_{3}}(\xi_{3})d\xi_{1}'d\xi_{3}'d\xi_{1}d\xi_{3}\Big)$$
minus 
$$\sigma_{n}(\xi)\int_{\mathbb R^4}\sigma_{n_{1}}(\xi_{1})\frac{e^{-it(\mu_{1}+\mu_{2})}}{\mu_{1}(\mu_{1}+\mu_{2})}\partial_{t}\Big(\hat{v}_{m_{1}}(\xi_{1}')\hat{\bar{v}}_{m_{2}}(\xi_{1}-\xi_{1}'-\xi_{3}')\hat{v}_{m_{3}}(\xi_{3}')\hat{\bar{v}}_{n_{2}}(\xi-\xi_{1}-\xi_{3})\hat{v}_{n_{3}}(\xi_{3})\Big)d\xi_{1}'d\xi_{3}'d\xi_{1}d\xi_{3},$$
where $\mu_{1}=(\xi-\xi_{1})(\xi-\xi_{3})$ and $\mu_{2}=(\xi_{1}-\xi_{1}')(\xi_{1}-\xi_{3}').$ Equivalently,
\begin{equation}
\label{sms}
\mathcal F(q^{2,t}_{1,n})=\partial_{t}(\tilde q^{2,t}_{1,n})-\mathcal F(\tau^{2,t}_{1,n}).
\end{equation}
Thus, by doing the same at the remaining two sums of $Q^{2,t}_{n}$, namely $q^{2,t}_{2,n}, q^{2,t}_{3,n}$, we obtain the splitting 
\begin{equation}
\label{neq11}
\mathcal F(Q^{2,t}_{n})=\partial_{t}\mathcal F(\tilde Q^{2,t}_{n})-\mathcal F(T^{2,t}_{n}).
\end{equation}
These new operators $\tilde q^{2,t}_{i,n}$, $i=1,2,3$, act on the following "type" of sequences
$$\tilde q^{2,t}_{1,n}(v_{m_{1}},\bar{v}_{m_{2}},v_{m_{3}},\bar{v}_{n_{2}},v_{n_{3}}),$$
with $m_{1}-m_{2}+m_{3}\approx n_{1}$ and $n_{1}-n_{2}+n_{3}\approx n$,
$$\tilde q^{2,t}_{2,n}(v_{n_{1}},\bar{v}_{m_{1}},v_{m_{2}},\bar{v}_{m_{3}},v_{n_{3}}),$$
with $m_{1}-m_{2}+m_{3}\approx n_{2}$ and $n_{1}-n_{2}+n_{3}\approx n$, and
$$\tilde q^{2,t}_{3,n}(v_{n_{1}}\bar{v}_{n_{2}},v_{m_{1}},\bar{v}_{m_{2}},v_{m_{3}}),$$
with $m_{1}-m_{2}+m_{3}\approx n_{3}$ and $n_{1}-n_{2}+n_{3}\approx n$. 

In order to proceed we need a similar lemma for the operator $\tilde Q^{2,t}_{n}$ as the one we had for $\tilde Q^{1,t}_{n}$ (see Lemma \ref{fir}). Here we state it only for $\tilde q^{2,t}_{1,n}$ (remember that we look only at frequencies on $A_{N}(n)^{c}$ and $C_{1}^{c}$):

\begin{lemma}
\label{fir34}
\begin{equation}
\|\tilde q^{2,t}_{1,n}(v_{m_{1}},\bar{v}_{m_{2}},v_{m_{3}},\bar{v}_{n_{2}},v_{n_{3}})\|_{2}\lesssim\frac{\|v_{m_{1}}\|_{2}\|v_{m_{2}}\|_{2}\|v_{m_{3}}\|_{2}\|v_{n_{2}}\|_{2}\|v_{n_{3}}\|_{2}}{|n-n_{1}||n-n_{3}||(n-n_{1})(n-n_{3})+(n_{1}-m_{1})(n_{1}-m_{3})|}.
\end{equation}
\end{lemma}
\begin{proof}
Writing out the Fourier transforms of the functions inside the integral of $\mathcal F(\tilde q^{2,t}_{1,n})$ it is not hard to see that 
$$\mathcal F(\tilde q^{2,t}_{1,n}(v_{m_{1}},\bar{v}_{m_{2}},v_{m_{3}},\bar{v}_{n_{2}},v_{n_{3}}))(\xi)=e^{-it\xi^{2}} \mathcal F(R^{2,t}_{n,n_{1}}(u_{m_{1}},\bar{u}_{m_{2}},u_{m_{3}},\bar{u}_{n_{2}},u_{n_{3}}))(\xi),$$
where the operator 
\begin{equation}
\label{set2}
R^{2,t}_{n,n_{1}}(w_{m_{1}},\bar{w}_{m_{2}},w_{m_{3}},\bar{w}_{n_{2}},w_{n_{3}})(x)=
\end{equation}
$$\int_{\mathbb R^5}K^{(2)}_{n,n_{1}}(x,x_{1}',y',x_{3}',y,x_{3})w_{m_{1}}(x_{1}')\bar{w}_{m_{2}}(y')w_{m_{3}}(x_{3}')\bar{w}_{n_{2}}(y)w_{n_{3}}(x_{3})\ dx_{1}'dy'dx_{3}'dydx_{3}$$
and the Kernel $K^{(2)}_{n,n_{1}}$ is given by the formula
\begin{equation}
\label{set3}
K^{(2)}_{n,n_{1}}(x,x_{1}',y',x_{3}',y,x_{3})=
\end{equation}
$$\int_{\mathbb R^5}[e^{i\xi_{1}'(x-x_{1}')+i\eta'(x-y')+i\xi_{3}'(x-x_{3}')+i\eta(x-y)+i\xi_{3}(x-x_{3})}]$$
$$\frac{\sigma_{n}(\xi_{1}'+\eta'+\xi_{3}'+\eta+\xi_{3})\sigma_{n_{1}}(\xi_{1}'+\eta'+\xi_{3}')}{(\eta+\eta'+\xi_{1}'+\xi_{3}')(\eta+\xi_{3})[(\eta+\eta'+\xi_{1}'+\xi_{3}')(\eta+\xi_{3})+(\eta'+\xi_{1}')(\eta'+\xi_{3}')]} d\xi_{1}'d\eta'd\xi_{3}'d\eta d\xi_{3}=$$
$$(\mathcal F^{-1}\rho^{(2)}_{n,n_{1}})(x-x_{1}',x-y',x-x_{3}',x-y,x-x_{3}),$$
and the function $\rho^{(2)}_{n,n_{1}}$ equals
$$\rho^{(2)}_{n,n_{1}}(\xi_{1}',\eta',\xi_{3}',\eta,\xi_{3})=\frac{\sigma_{n}(\xi_{1}'+\eta'+\xi_{3}'+\eta+\xi_{3})\sigma_{n_{1}}(\xi_{1}'+\eta'+\xi_{3}')}{(\eta+\eta'+\xi_{1}'+\xi_{3}')(\eta+\xi_{3})[(\eta+\eta'+\xi_{1}'+\xi_{3}')(\eta+\xi_{3})+(\eta'+\xi_{1}')(\eta'+\xi_{3}')]}.$$
The operator $R^{2,t}_{n,n_{1}}$ is estimated in $L^2$ as in the proof of Lemma \ref{fir} and the function $\rho^{(2)}_{n,n_{1}}$ plays the same role as the function $\rho^{(1)}_{n}$ did for $R^{1,t}_{n}$, therefore, 
$$\|R^{2,t}_{n,n_{1}}(v_{m_{1}},\bar{v}_{m_{2}},v_{m_{3}},\bar{v}_{n_{2}},v_{n_{3}})\|_{2}\lesssim\|\rho^{(2)}_{n,n_{1}}\|_{\infty}\|v_{m_{1}}\|_{2}\|v_{m_{2}}\|_{2}\|v_{m_{3}}\|_{2}\|v_{n_{2}}\|_{2}\|v_{n_{3}}\|_{2},$$
and since $\xi_{1}'\approx m_{1}, \eta'\approx-m_{2}, \xi_{3}'\approx m_{3}, \eta\approx-n_{2}, \xi_{3}\approx n_{3}$ we obtain
$$\|\rho^{(2)}_{n,n_{1}}\|_{\infty}\lesssim\frac1{|n-n_{1}||n-n_{3}||(n-n_{1})(n-n_{3})+(n_{1}-m_{1})(n_{1}-m_{3})|},$$
which finishes the proof.
\end{proof}

\begin{remark}
The operator $\tilde q^{2,t}_{3,n}$ satisfies exactly the same bound as $\tilde q^{2,t}_{1,n}$ since the only difference between these operators is a permutation of their variables. On the other hand, the operator $\tilde q^{2,t}_{2,n}$ is a bit different, since instead of taking only the permutation we have to conjugate the $2$nd variable too. Thus, a similar argument as the one given in Lemma \ref{fir34} leads to the estimate
\begin{equation}
\label{fir354}
\|\tilde q^{2,t}_{2,n}(v_{n_{1}},\bar{v}_{m_{1}},v_{m_{2}},\bar{v}_{m_{3}},v_{n_{3}})\|_{2}\lesssim\frac{\|v_{n_{1}}\|_{2}\|v_{m_{1}}\|_{2}\|v_{m_{2}}\|_{2}\|v_{m_{3}}\|_{2}\|v_{n_{3}}\|_{2}}{|(n-n_{1})(n-n_{3})||(n-n_{1})(n-n_{3})-(n_{2}-m_{1})(n_{2}-m_{3})|}
\end{equation}
which is not exactly the same as the one we had for the operators $\tilde q^{2,t}_{1,n}, \tilde q^{2,t}_{3,n}$ since in the denominator instead of having $\mu_{1}+\mu_{2}$ we have $\mu_{1}-\mu_{2}$ ($\mu_{1}=(n-n_{1})(n-n_{3})$ and in the first case $\mu_{2}=(n_{1}-m_{1})(n_{1}-\mu_{3})$, $m_{1}, m_{3}$ being the "children" of $n_{1}$, whereas in the second case $\mu_{2}=(n_{2}-m_{1})(n_{2}-m_{3})$, $m_{1}, m_{3}$ being the "children" of $n_{2}$). It is readily checked that this change in the sign does not really affect the calculations that are to follow.
\end{remark}
This Lemma allows us to move forward with our iteration process and show that the operators

\begin{equation}
\label{formm1}
N_{0}^{(3)}(v)(n):=\sum_{A_{N}(n)^{c}}\sum_{C_{1}^{c}}\tilde Q^{2,t}_{n}=\sum_{A_{N}(n)^{c}}\sum_{C_{1}^{c}}\sum_{i=1}^{3}\tilde q^{2,t}_{i,n}
\end{equation}
and
\begin{equation}
\label{formm2}
N^{(3)}_{r}(v)(n):=\sum_{A_{N}(n)^{c}}\sum_{C_{1}^{c}}\Big(\tilde q^{2,t}_{1,n}(R^{t}_{2}(v)(m_{1})-R^{t}_{1}(v)(m_{1}),\bar{v}_{m_{2}},v_{m_{3}},\bar{v}_{n_{2}},v_{n_{3}})+
\end{equation}
$$\tilde q^{2,t}_{1,n}(v_{m_{1}},\overline{R^{t}_{2}(v)(m_{2})-R^{t}_{1}(v)(m_{2})},v_{m_{3}},\bar{v}_{n_{2}},v_{n_{3}})+\ldots+\tilde q^{2,t}_{3,n}(v_{n_{1}}\bar{v}_{n_{2}},v_{m_{1}},\bar{v}_{m_{2}},R^{t}_{2}(v)(m_{3})-R^{t}_{1}(v)(m_{3}))\Big),$$
are bounded on $l^{q}L^{2}$. The operator $N_{r}^{(3)}$ appears when we substitute each of the derivatives in the operator $\sum_{i=1}^{3}\tau^{2,t}_{i,n}$ by the expression given in (\ref{mainmain}). Notice that the operator $N_{0}^{(3)}$ has three summands and the operator $N_{r}^{(3)}$ has $3\cdot 5=15$ summands. Here is the claim:

\begin{lemma}
\label{gg}
$$\|N_{0}^{(3)}(v)\|_{l^{q}L^{2}}\lesssim N^{-2+\frac1{100}+\frac2{q'}-\frac1{100q'}+}\|v\|_{M_{2,q}}^{5},$$
and 
$$\|N_{0}^{(3)}(v)-N_{0}^{(3)}(w)\|_{l^{q}L^{2}}\lesssim N^{-2+\frac1{100}+\frac2{q'}-\frac1{100q'}+}(\|v\|_{M_{2,q}}^{4}+\|w\|_{M_{2,q}}^{4})\|v-w\|_{M_{2,q}}.$$

$$\|N_{r}^{(3)}(v)\|_{l^{q}L^{2}}\lesssim N^{-2+\frac1{100}+\frac2{q'}-\frac1{100q'}+}\|v\|^{7}_{M_{2,q}},$$
and
$$\|N_{r}^{(3)}(v)-N_{r}^{(3)}(w)\|_{l^{q}L^{2}}\lesssim N^{-2+\frac1{100}+\frac2{q'}-\frac1{100q'}+}(\|v\|_{M_{2,q}}^{6}+\|w\|_{M_{2,q}}^{6})\|v-w\|_{M_{2,q}}.$$
\end{lemma}
\begin{proof}
Let us start with the operator $N_{0}^{(3)}$ and for simplicity of the presentation we will consider only the sum with the term $\tilde q^{2,t}_{1,n}$. As in the proof of Lemma \ref{fir3} we have from (\ref{num}) that for fixed $n$ and $\mu_{1}$ there are at most $o(|\mu_{1}|^{+})$ many choices for $n_{1}, n_{2}, n_{3}$ (such that $(n-n_{1})(n-n_{3})=\mu_{1}$) and for fixed $n_{1}$ and $\mu_{2}$ there are at most $o(|\mu_{2}|^{+})$ many choices for $m_{1}, m _{2}, m_{3}$ (such that $(n_{1}-m_{1})(n_{1}-m_{3})=\mu_{2}$). Thus, from Lemma \ref{fir34} we obtain
$$\sum_{A_{N}(n)^{c}}\sum_{C_{1}^{c}}\|\tilde q^{2,t}_{1,n}(v_{m_{1}},\bar{v}_{m_{2}},v_{m_{3}},\bar{v}_{n_{2}},v_{n_{3}})\|_{2}\lesssim$$
$$\sum_{A_{N}(n)^{c}}\sum_{C_{1}^{c}}\frac{\|v_{m_{1}}\|_{2}\|v_{m_{2}}\|_{2}\|v_{m_{3}}\|_{2}\|v_{n_{2}}\|_{2}\|v_{n_{3}}\|_{2}}{|n-n_{1}||n-n_{3}||(n-n_{1})(n-n_{3})+(n_{1}-m_{1})(n_{1}-m_{3})|}$$
and the RHS is equal to 
$$\sum_{A_{N}(n)^{c}}\sum_{C_{1}^{c}}\frac{\|v_{m_{1}}\|_{2}\|v_{m_{2}}\|_{2}\|v_{m_{3}}\|_{2}\|v_{n_{2}}\|_{2}\|v_{n_{3}}\|_{2}}{|\mu_{1}||\mu_{1}+\mu_{2}|}$$
which by H\"older's inequality is bounded above by
$$\Big(\sum_{A_{N}(n)^{c}}\sum_{C_{1}^{c}}\frac1{|\mu_{1}|^{q'}|\mu_{1}+\mu_{2}|^{q'}}|\mu_{1}|^{+}|\mu_{2}|^{+}\Big)^{\frac1{q'}}\Big(\sum_{A_{N}(n)^{c}}\sum_{C_{1}^{c}}\|v_{m_{1}}\|_{2}^{q}\|v_{m_{2}}\|_{2}^{q}\|v_{m_{3}}\|_{2}^{q}\|v_{n_{2}}\|_{2}^{q}\|v_{n_{3}}\|_{2}^{q}\Big)^{\frac1{q}}.$$
By a very crude estimate it is not difficult to see that the first sum behaves like the number $N^{-2+\frac1{100}+\frac2{q'}-\frac1{100q'}+}$. Then, by taking the $l^{q}$ norm and applying Young's inequality for convolutions we are done. For the operator $N_{r}^{(3)}$ the proof is the same but in addition we use Lemma \ref{lem} for the operator $R_{2}^{t}-R_{1}^{t}$. 
\end{proof}
The operator that remains to be estimated is defined as
\begin{equation}
\label{formm3}
N^{(3)}(v)(n):=\sum_{A_{N}(n)^{c}}\sum_{C_{1}^{c}}\Big(\tilde q^{2,t}_{1,n}(N_{1}^{t}(v)(m_{1}),\bar{v}_{m_{2}},v_{m_{3}},\bar{v}_{n_{2}},v_{n_{3}})+
\end{equation}
$$\tilde q^{2,t}_{1,n}(v_{m_{1}},\overline{N_{1}^{t}(v)(m_{2})},v_{m_{3}},\bar{v}_{n_{2}},v_{n_{3}})+\ldots+\tilde q^{2,t}_{3,n}(v_{n_{1}}\bar{v}_{n_{2}},v_{m_{1}},\bar{v}_{m_{2}},N_{1}^{t}(v)(m_{3}))\Big),$$
which is the same as $N_{r}^{(3)}$ but in the place of the operator $R_{2}^{t}-R_{1}^{t}$ we have $N_{1}^{t}.$ As before, we write
\begin{equation}
\label{formm4}
N^{(3)}=N_{1}^{(3)}+N_{2}^{(3)},
\end{equation}
where $N_{1}^{(3)}$ is the restriction of $N^{(3)}$ onto the set of frequencies
\begin{equation}
\label{setset2}
C_{2}=\{|\tilde\mu_{3}|\leq 7^{3}|\tilde\mu_{2}|^{1-\frac1{100}}\}\cup\{|\tilde\mu_{3}|\leq 7^{3}|\mu_{1}|^{1-\frac1{100}}\},
\end{equation}
where $\tilde\mu_{2}=\mu_{1}+\mu_{2}$ and $\tilde\mu_{3}=\mu_{1}+\mu_{2}+\mu_{3}.$ The following is true:

\begin{lemma}
\label{gg1}
$$\|N_{1}^{(3)}(v)\|_{l^{q}L^{2}}\lesssim N^{-2+\frac1{100}+\frac3{q'}-\frac2{100q'}+}\|v\|_{M_{2,q}}^{7},$$
and
$$\|N_{1}^{(3)}(v)-N_{1}^{(3)}(w)\|_{l^{q}L^{2}}\lesssim N^{-2+\frac1{100}+\frac3{q'}-\frac2{100q'}+}(\|v\|^{6}_{M_{2,q}}+\|w\|^{6}_{M_{2,q}})\|v-w\|_{M_{2,q}}.$$
\end{lemma}
\begin{proof}
Let us only consider the very first summand of the operator $N_{1}^{(3)}$, that is the operator $\tilde q_{1,n}^{2,t}$ with $N_{1}^{t}$ acting on its first variable, since for the other summands similar considerations apply. For the proof we use again the divisor counting argument. From (\ref{num}) it follows that for fixed $n$ and $\mu_{1}$ there are at most $o(|\mu_{1}|^{+})$ many choices for $n_{1}, n_{2}, n_{3}$ ($\mu_{1}=(n-n_{1})(n-n_{3})$, $n\approx n_{1}-n_{2}+n_{3}$). For fixed $n_{1}$ and $\mu_{2}$ there are at most $o(|\mu_{2}|^{+})$ many choices for $m_{1}, m_{2}, m_{3}$ ($\mu_{2}=(n_{1}-m_{1})(n_{1}-m_{3})$, $n_{1}\approx m_{1}-m_{2}+m_{3}$) and for fixed $m_{1}$ and $\mu_{3}$ there are at most $o(|\mu_{3}|^{+})$ many choices for $k_{1}, k_{2}, k_{3}$ ($\mu_{3}=(m_{1}-k_{1})(m_{1}-k_{3})$, $m_{1}\approx k_{1}-k_{2}+k_{3}$). 

First, let us assume that our frequencies satisfy $|\tilde\mu_{3}|\lesssim |\tilde\mu_{2}|^{1-\frac1{100}}$. Since, $\tilde\mu_{3}=\tilde\mu_{2}+\mu_{3}$ we have $|\mu_{3}|\sim|\tilde\mu_{2}|$. Moreover, for fixed $|\tilde\mu_{2}|$ (equivalently, for fixed $\mu_{1}, \mu_{2}$) there are at most $O(|\tilde\mu_{2}|^{1-\frac1{100}})$ many choices for $\tilde\mu_{3}$ and hence, for $\mu_{3}=\tilde\mu_{3}-\tilde\mu_{2}$. In addition, $|\mu_{2}|\lesssim\max(|\mu_{1}|, |\tilde\mu_{2}|)$ and we should recall that since we are on $C_{1}^{c}$ we have $|\tilde\mu_{2}|=|\mu_{1}+\mu_{2}|>5^{3}|\mu_{1}|^{1-\frac1{100}}>5^{3}N^{1-\frac1{100}}$. Then, the expression
$$\sum_{A_{N}(n)^{c}}\sum_{C_{1}^{c}}\sum_{C_{2}}\|\tilde q^{2,t}_{1,n}(Q^{1,t}_{m_{1}}(v_{k_{1}},\bar{v}_{k_{2}},v_{k_{3}}),\bar{v}_{m_{2}},v_{m_{3}},\bar{v}_{n_{2}},v_{n_{3}})\|_{2}$$
with the use of Lemma \ref{fir34} and a trivial bound of the operator $Q^{1,t}_{m_{1}}$ in $L^{2}$ (see proof of Lemma \ref{lemle}) we obtain the upper bound 
$$\sum_{A_{N}(n)^{c}}\sum_{C_{1}^{c}}\sum_{C_{2}}\frac{\|v_{k_{1}}\|_{2}\|v_{k_{2}}\|_{2}\|v_{k_{3}}\|_{2}\|v_{m_{2}}\|_{2}\|v_{m_{3}}\|_{2}\|v_{n_{2}}\|_{2}\|v_{n_{3}}\|_{2}}{|n-n_{1}||n-n_{3}||(n-n_{1})(n-n_{3})+(n_{1}-m_{1})(n_{1}-m_{3})|}=$$
$$\sum_{A_{N}(n)^{c}}\sum_{C_{1}^{c}}\sum_{C_{2}}\frac{\|v_{k_{1}}\|_{2}\|v_{k_{2}}\|_{2}\|v_{k_{3}}\|_{2}\|v_{m_{2}}\|_{2}\|v_{m_{3}}\|_{2}\|v_{n_{2}}\|_{2}\|v_{n_{3}}\|_{2}}{|\mu_{1}||\tilde\mu_{2}|}$$
and by H\"older's inequality we obtain
\begin{equation}
\label{usus}
\Big(\sum_{\substack{|\mu_{1}|>N \\ |\tilde\mu_{2}|>5^{3}N^{1-\frac1{100}}}}\frac{|\mu_{1}|^{+}|\mu_{2}|^{+}|\mu_{3}|^{+}|\tilde\mu_{2}|^{1-\frac1{100}}}{|\mu_{1}|^{q'}|\tilde\mu_{2}|^{q'}}\Big)^{\frac1{q'}}\times
\end{equation}
$$\Big(\sum_{A_{N}(n)^{c}}\sum_{C_{1}^{c}}\sum_{C_{2}}\|v_{k_{1}}\|_{2}^{q}\|v_{k_{2}}\|_{2}^{q}\|v_{k_{3}}\|_{2}^{q}\|v_{m_{2}}\|_{2}^{q}\|v_{m_{3}}\|_{2}^{q}\|v_{n_{2}}\|_{2}^{q}\|v_{n_{3}}\|_{2}^{q}\Big)^{\frac1{q}}.$$
The first sum is bounded above by
\begin{equation}
\label{estst}
\Big(\sum_{\substack{|\mu_{1}|>N \\ |\tilde\mu_{2}|>5^{3}N^{1-\frac1{100}}}}\frac1{|\mu_{1}|^{q'-\epsilon}|\tilde\mu_{2}|^{q'-1+\frac1{100}-\epsilon}}\Big)^{\frac1{q'}}\lesssim \Big(N^{3(1-\frac1{100})-q'(2-\frac1{100})+\frac1{100^{2}}+}\Big)^{\frac1{q'}}
\end{equation}
and by the use of Young's inequality at the second sum we are done. 

On the other hand, if $|\tilde\mu_{3}|\lesssim|\mu_{1}|^{1-\frac1{100}}$, then for fixed $\mu_{1}, \mu_{2}$ there are at most $O(|\mu_{1}|^{1-\frac1{100}})$ many choices for $\tilde\mu_{3}$ and hence for $\mu_{3}.$ After this observation, the calculations are exactly the same as before but the first sum of (\ref{usus}) becomes
\begin{equation}
\label{estst1}
\Big(\sum_{\substack{|\mu_{1}|>N \\ |\tilde\mu_{2}|>5^{3}N^{1-\frac1{100}}}}\frac1{|\mu_{1}|^{q'-1+\frac1{100}-\epsilon}|\tilde\mu_{2}|^{q'-\epsilon}}\Big)^{\frac1{q'}}\lesssim \Big(N^{3-\frac2{100}-q'(2-\frac1{100})+}\Big)^{\frac1{q'}}.
\end{equation}
Between the two exponents of $N$ in (\ref{estst}) and (\ref{estst1}) we see that (\ref{estst1}) is the dominating one and the proof is complete.
\end{proof}
To the remaining part, namely $N_{2}^{(3)}$, we have to apply the differentiation by parts technique again. Note that here we only look at frequencies such that
$$|\tilde\mu_{3}|=|\mu_{1}+\mu_{2}+\mu_{3}|>7^{3}|\mu_{1}|^{1-\frac1{100}}>7^{3}N^{1-\frac1{100}},$$
or equivalently, frequencies that are on the set $C_{2}^{c}$. Instead, we will present the general $J$th step of the iteration procedure and prove the required lemmata. To do this, we need to use the tree notation as it was introduced in \cite{GKO}.

\subsection{The Tree Notation and the Induction Step}
A tree $T$ is a finite, partially ordered set with the following properties:
\begin{itemize}
\item For any $a_{1}, a_{2}, a_{3}, a_{4}\in T$ if $a_{4}\leq a_{2}\leq a_{1}$ and $a_{4}\leq a_{3}\leq a_{1}$ then $a_{2}\leq a_{3}$ or $a_{3}\leq a_{2}$. 
\item There exists a maximum element $r\in T$, that is $a\leq r$ for all $a\in T$ which is called the root. 
\end{itemize}
We call the elements of $T$ the nodes of the tree and in this content we will say that $b\in T$ is a child of $a\in T$ (or equivalently, that $a$ is the parent of $b$) if $b\leq a, b\neq a$ and for all $c\in T$ such that $b\leq c\leq a$ we have either $b=c$ or $c=a$. 

A node $a\in T$ is called terminal if it has no children. A nonterminal node $a\in T$ is a node with exactly $3$ children $a_{1}$, the left child, $a_{2}$, the middle child, and $a_{3}$, the right child. We define the sets
\begin{equation}
\label{setsetset}
T^{0}=\{\mbox{all nonterminal nodes}\},
\end{equation}
and
\begin{equation}
\label{setsetset1}
T^{\infty}=\{\mbox{all terminal nodes}\}.
\end{equation}
Obviously, $T=T^{0}\cup T^{\infty}$, $T^{0}\cap T^{\infty}=\emptyset$ and if $|T^{0}|=j\in\Z_{+}$ we have $|T|=3j+1$ and $|T^{\infty}|=2j+1.$ We denote the collection of trees with $j$ parental nodes by
\begin{equation}
\label{setsetset2}
T(j)=\{T\ \mbox{is a tree with}\ |T|=3j+1\}.
\end{equation}
Next, we say that a sequence of trees $\{T_{j}\}_{j=1}^{J}$ is a chronicle of $J$ generations if:
\begin{itemize}
\item $T_{j}\in T(j)$ for all $j=1, 2, \ldots, J$.
\item $T_{j+1}$ is obtained by changing one of the terminal nodes of $T_{j}$ into a nonterminal node with exactly $3$ children, for all $j=1, 2, \ldots, J-1$.
\end{itemize}
Let us also denote by $\mathcal I(J)$ the collection of trees of the $J$th generation. It is easily checked by an induction argument that
\begin{equation}
\label{setsetset3}
|\mathcal I(J)|=1\cdot 3\cdot 5\ldots(2J-1)=:(2J-1)!!.
\end{equation}
Given a chronicle $\{T_{j}\}_{j=1}^{J}$ of $J$ generations we refer to $T_{J}$ as an ordered tree of the $J$th generation. We should keep in mind that the notion of ordered trees comes with associated chronicles. It includes not only the shape of the tree but also how it "grew". 

Given an ordered tree $T$ we define an index function $n:T\to\Z$ such that
\begin{itemize}
\item $n_{a}\approx n_{a_{1}}-n_{a_{2}}+n_{a_{3}}$ for all $a\in T^{0}$, where $a_{1}, a_{2}, a_{3}$ are the children of $a$,
\item $n\not\approx n_{a_{1}}$ and $n\not\approx n_{a_{3}}$, for all $a\in T^{0}$,
\item $|\mu_{1}|:=2|n_{r}-n_{r_{1}}||n_{r}-n_{r_{3}}|>N$, where $r$ is the root of $T$,
\end{itemize}
and we denote the collection of all such index functions by $\mathcal R(T)$. 

For the sake of completeness, as it was done in \cite{GKO}, given an ordered tree $T$ with the chronicle $\{T_{j}\}_{j=1}^{J}$ and associated index functions $n\in\mathcal R(T)$, we need to keep track of the generations of frequencies. Fix an $n\in\mathcal R(T)$ and consider the very first tree $T_{1}$. Its nodes are the root $r$ and its children $r_{1}, r_{2}, r_{3}$. We define the first generation of frequencies by 
$$(n^{(1)},n_{1}^{(1)},n_{2}^{(1)},n_{3}^{(1)}):=(n_{r},n_{r_{1}},n_{r_{2}},n_{r_{3}}).$$
From the definition of the index function we have
$$n^{(1)}\approx n_{1}^{(1)}-n_{2}^{(1)}+n_{3}^{(1)},\ n_{1}^{(1)}\not\approx n^{(1)}\not\approx n_{3}^{(1)}.$$
The ordered tree $T_{2}$ of the second generation is obtained from $T_{1}$ by changing one of its terminal nodes $a=r_{k}\in T_{1}^{\infty}$ for some $k=1,2,3$ into a nonterminal node. Then, the second generation of frequencies is defined by
$$(n^{(2)},n_{1}^{(2)},n_{2}^{(2)},n_{3}^{(2)}):=(n_{a},n_{a_{1}},n_{a_{2}},n_{a_{3}}).$$
Thus, we have $n^{(2)}=n_{k}^{(1)}$ for some $k=1,2,3$ and from the definition of the index function we have
$$n^{(2)}\approx n_{1}^{(2)}-n_{2}^{(2)}+n_{3}^{(2)},\ n_{1}^{(2)}\not\approx n^{(2)}\not\approx n_{3}^{(2)}.$$
This should be compared with what happened in the calculations we presented before when passing from the first step of the iteration process into the second step. Every time we apply the differentiation by parts technique we introduce a new set of frequencies. 

After $j-1$ steps, the ordered tree $T_{j}$ of the $j$th generation is obtained from $T_{j-1}$ by changing one of its terminal nodes $a\in T_{j-1}^{\infty}$ into a nonterminal node. Then, the $j$th generation frequencies are defined as 
$$(n^{(j)},n_{1}^{(j)},n_{2}^{(j)},n_{3}^{(j)}):=(n_{a},n_{a_{1}},n_{a_{2}},n_{a_{3}}),$$
and we have $n^{(j)}=n_{k}^{(m)}(=n_{a})$ for some $m=1,2,\ldots,j-1$ and $k=1,2,3$, since this corresponds to the frequency of some terminal node in $T_{j-1}$. In addition, from the definition of the index function we have
$$n^{(j)}\approx n_{1}^{(j)}-n_{2}^{(j)}+n_{3}^{(j)},\ n_{1}^{(j)}\not\approx n^{(j)}\not\approx n_{3}^{(j)}.$$
Finally, we use $\mu_{j}$ to denote the corresponding phase factor introduced at the $j$th generation. That is,
\begin{equation}
\label{muuu}
\mu_{j}=2(n^{(j)}-n_{1}^{(j)})(n^{(j)}-n_{3}^{(j)}),
\end{equation}
and we also introduce the quantities
\begin{equation}
\label{qqq}
\tilde\mu_{J}=\sum_{j=1}^{J}\mu_{j},\ \hat{\mu}_{J}=\prod_{j=1}^{J}\tilde\mu_{j}.
\end{equation}
We should keep in mind that everytime we apply differentiation by parts and split the operators, we need to control the new frequencies that arise from this procedure. For this reason we need to define the sets (see (\ref{setset1}) and (\ref{setset2})):
\begin{equation}
\label{sesee}
C_{J}:=\{|\tilde\mu_{J+1}|\leq(2J+3)^{3}|\tilde\mu_{J}|^{1-\frac1{100}}\}\cup\{|\tilde\mu_{J+1}|\leq(2J+3)^{3}|\mu_{1}|^{1-\frac1{100}}\}.
\end{equation}

Let us see how to use this notation and terminology in our calculations. On the very first step, $J=1$, we have only one tree, the root node $r$ and its three children $r_{1}, r_{2}, r_{3}$ (sometimes, when it is clear from the context, we will identify the nodes and the frequencies assigned to them, that is, we have the root $n=n_{r}$ and its three children $n_{r_{1}}=n_{1}, n_{r_{2}}=n_{2}, n_{r_{3}}=n_{3}$) and we have only one operator that needs to be controlled in order to proceed further, namely $\tilde q^{1,t}_{n}:=\tilde Q^{1,t}_{n}$. 

On the second step, $J=2$, we have three operators $\tilde q^{2,t}_{n,n_{1}}:=\tilde q^{2,t}_{1,n}, \tilde q^{2,t}_{n,n_{2}}:=\tilde q^{2,t}_{2,n}, \tilde q^{2,t}_{n,n_{3}}:=\tilde q^{2,t}_{3,n}$ that play the same role as $\tilde q^{1,t}_{n}$ did for the first step. Let us observe that for each one of these operators we must have estimates on their $L^{2}$ norms in order to be able and continue the iteration. These estimates were provided by Lemmata \ref{fir} and \ref{fir34}. 

On the general $J$th step we will have $|\mathcal I(J)|$ operators of the $\tilde q^{J,t}_{T^0,\mathbf n}$ "type" each one corresponding to one of the ordered trees of the $J$th generation, $T\in T(J)$, where $\mathbf n$ is an arbitrary fixed index function on $T$. We have the subindices $T^0$ and $\mathbf n$ because each one of these operators has Fourier transform supported on the cubes with centers the frequencies assigned to the nodes that belong to $T^0$. 

Let us denote by $T_{\alpha}$ all the nodes of the ordered tree $T$ that are descendants of the node $\alpha\in T^{0}$, i.e. $T_{\alpha}=\{\beta\in T:\beta\leq\alpha,\ \beta\neq\alpha\}$.

We also need to define the principal and final "signs" of a node $a\in T$ which are functions from the tree $T$ into the set $\{\pm1\}$:
\begin{equation}
\label{signsign}
\mbox{psgn}(a)=\begin{cases}
+1,\ a\ \mbox{is not the middle child of his father}\\
+1,\ a=r,\ \mbox{the root node}\\
-1,\ a\ \mbox{is the middle child of his father}
\end{cases}
\end{equation}
\begin{equation}
\label{signsignsign}
\mbox{fsgn}(a)=\begin{cases}
+1,\ \mbox{psgn}(a)=+1\ \mbox{and}\ a\ \mbox{has an even number of middle predecessors}\\
-1,\ \mbox{psgn}(a)=+1\ \mbox{and}\ a\ \mbox{has an odd number of middle predecessors}\\
-1,\ \mbox{psgn}(a)=-1\ \mbox{and}\ a\ \mbox{has an even number of middle predecessors}\\
+1,\ \mbox{psgn}(a)=-1\ \mbox{and}\ a\ \mbox{has an odd number of middle predecessors},
\end{cases}
\end{equation}
where the root node $r\in T$ is not considered a middle father. 

The operators $\tilde q^{J,t}_{T^0,\mathbf n}$ are defined through their Fourier transforms as
\begin{equation}
\label{oops}
\mathcal F(\tilde q^{J,t}_{T^0,\mathbf n}(\{w_{n_\beta}\}_{\beta\in T^{\infty}}))(\xi)=e^{-it\xi^{2}}\mathcal F(R^{J,t}_{T^0,\mathbf n}(\{e^{-it\partial_{x}^{2}}w_{n_\beta}\}_{\beta\in T^{\infty}}))(\xi),
\end{equation}
where the operator $R^{J,t}_{T^0,\mathbf n}$ acts on the functions $\{w_{n_\beta}\}_{\beta\in T^{\infty}}$ as
\begin{equation}
\label{oops1}
R^{J,t}_{T^0,\mathbf n}(\{w_{n_\beta}\}_{\beta\in T^{\infty}})(x)=\int_{\R^{2J+1}}K^{(J)}_{T^0}(x,\{x_{\beta}\}_{\beta\in T^{\infty}})\Big[\otimes_{\beta\in T^{\infty}}w_{n_\beta}(x_{\beta})\Big]\ \prod_{\beta\in T^{\infty}} dx_{\beta},
\end{equation}
and the Kernel $K^{(J)}_{T^0,\mathbf n}$ is defined as 
\begin{equation}
\label{oopss}
K^{(J)}_{T^0,\mathbf n}(x,\{x_{\beta}\}_{\beta\in T^{\infty}})=\mathcal F^{-1}(\rho^{(J)}_{T^{0},\mathbf n})(\{x-x_{\beta}\}_{\beta\in T^{\infty}}).
\end{equation}
Here is the formula for the function $\rho^{(J)}_{T^0,\mathbf n}$ with ($|T^{\infty}|=2J+1$)-variables, $\xi_{\beta}$, $\beta\in T^{\infty}$:
\begin{equation}
\label{jc}
\rho^{(J)}_{T^0,\mathbf n}(\{\xi_{\beta}\}_{\beta\in T^{\infty}})=\Big[\prod_{\alpha\in T^0}\sigma_{n_{\alpha}}\Big(\sum_{\beta\in T^{\infty}\cap T_{\alpha}}\mbox{fsgn}(\beta)\ \xi_{\beta}\Big)\Big]\frac1{\hat{\mu}_{T}},
\end{equation}
where we denote by 
\begin{equation}
\label{yeah}
\hat{\mu}_{T}=\prod_{\alpha\in T^0}\tilde\mu_{\alpha},\ \tilde\mu_{\alpha}=\sum_{\beta\in T^{0}\setminus T_{\alpha}}\mu_{\beta},
\end{equation}
and for $\beta\in T^{0}$ we have
\begin{equation}
\label{yyeah}
\mu_{\beta}=2(\xi_{\beta}-\xi_{\beta_{1}})(\xi_{\beta}-\xi_{\beta_{3}}),
\end{equation}
where we impose the relation $\xi_{\alpha}=\xi_{\alpha_{1}}-\xi_{\alpha_{2}}+\xi_{\alpha_{3}}$ for every $\alpha\in T^{0}$ that appears in the calculations until we reach the terminal nodes of $T^{\infty}$. This is because in the definition of the function $ \rho^{J,t}_{T^0}$ we need the variables "$\xi$" to be assigned only at the terminal nodes of the tree $T$. We use the notation $\mu_{\beta}$ in similarity to $\mu_{j}$ of equation (\ref{muuu}) because this is the "continuous" version of the discrete case. In addition, the variables $\xi_{\alpha_{1}}, \xi_{\alpha_{2}}, \xi_{\alpha_{3}}$ that appear in the expression (\ref{jc}) are supported in such a way that $\xi_{\alpha_{1}}\approx n_{\alpha_{1}}, \xi_{\alpha_{2}}\approx n_{\alpha_{2}}, \xi_{\alpha_{3}}\approx n_{\alpha_{3}}$. This is because the functions $\sigma_{n_{\alpha}}$ are supported in such a way. Therefore, $|\hat{\mu}_{T}|\sim|\hat{\mu}_{J}|$. 

For the induction step of our iteration process it is easy to check that the following Lemma is true, which should be compared with Lemmata \ref{fir} and \ref{fir34}:

\begin{lemma}
\label{indu}
\begin{equation}
\|\tilde q^{J,t}_{T^0,\mathbf n}(\{v_{n_\beta}\}_{\beta\in T^{\infty}})\|_{2}\lesssim\Big(\prod_{\beta\in T^{\infty}}\|v_{n_\beta}\|_{2}\Big)\frac1{|\hat{\mu}_{T}|},
\end{equation}
for every tree $T\in T(J)$ and index function $\mathbf n\in\mathcal R(T)$.
\end{lemma}
Given an index function $\mathbf n$ and $2J+1$ functions $\{v_{n_\beta}\}_{\beta\in T^{\infty}}$ and $\alpha\in T^{\infty}$ we define the action of the operator $N_{1}^{t}$ (see (\ref{main10})) on the set $\{v_{n_\beta}\}_{\beta\in T^{\infty}}$ to be the same set as before but with the difference that we have substituted the function $v_{n_\alpha}$ by the new function $N_{1}^{t}(v)(n_\alpha)$. We will denote this new set of functions $N_{1}^{t,\alpha}(\{v_{n_\beta}\}_{\beta\in T^{\infty}})$. Similarly, the action of the operator $R_{2}^{t}-R_{1}^{t}$ (see (\ref{main9})) on the set of functions $\{v_{n_\beta}\}_{\beta\in T^{\infty}}$ will be denoted by $(R_{2}^{t,\alpha}-R_{1}^{t,\alpha})(\{v_{n_\beta}\}_{\beta\in T^{\infty}})$. 

The operator of the $J$th step, $J\geq 2$, that we want to estimate is given by the formula:
\begin{equation}
\label{fina}
N_{2}^{(J)}(v)(n):=\sum_{T\in T(J-1)}\sum_{\alpha\in T^{\infty}}\sum_{\substack{\mathbf n\in\mathcal R(T)\\ \mathbf n_{r}=n}}\tilde q^{J-1,t}_{T^0}(N_{1}^{t,\alpha}(\{v_{n_\beta}\}_{\beta\in T^{\infty}})).
\end{equation}
Applying differentiation by parts on the Fourier side (keep in mind that from the splitting procedure we are on the sets $A_{N}(n)^{c},C_{1}^{c},\ldots,C_{J-1}^{c}$) we obtain the expression
\begin{equation}
\label{fina1}
N_{2}^{(J)}(v)(n)=\partial_{t}(N_{0}^{(J+1)}(v)(n))+N_{r}^{(J+1)}(v)(n)+N^{(J+1)}(v)(n), 
\end{equation}
where
\begin{equation}
\label{fina2}
N_{0}^{(J+1)}(v)(n):=\sum_{T\in T(J)}\sum_{\substack{\mathbf n\in\mathcal R(T)\\ \mathbf n_{r}=n}}\tilde q^{J,t}_{T^0,\mathbf n}(\{v_{n_\beta}\}_{\beta\in T^{\infty}}),
\end{equation}
and
\begin{equation}
\label{fina3}
N_{r}^{(J+1)}(v)(n):=\sum_{T\in T(J)}\sum_{\alpha\in T^{\infty}}\sum_{\substack{\mathbf n\in\mathcal R(T)\\ \mathbf n_{r}=n}}\tilde q^{J,t}_{T^0,\mathbf n}((R^{t,\alpha}_{2}-R^{t,\alpha}_{1})(\{v_{n_{\beta}}\}_{\beta\in T^{\infty}})),
\end{equation}
and
\begin{equation}
\label{fina4}
N^{(J+1)}(v)(n):=\sum_{T\in T(J)}\sum_{\alpha\in T^{\infty}}\sum_{\substack{\mathbf n\in\mathcal R(T)\\ \mathbf n_{r}=n}}\tilde q^{J,t}_{T^0,\mathbf n}(N_{1}^{t,\alpha}(\{v_{n_{\beta}}\}_{\beta\in T^{\infty}})).
\end{equation}
We also split the operator $N^{(J+1)}$ as the sum
\begin{equation}
\label{fina5}
N^{(J+1)}=N_{1}^{(J+1)}+N_{2}^{(J+1)},
\end{equation}
where $N_{1}^{(J+1)}$ is the restriction of $N^{(J+1)}$ onto $C_{J}$ and $N_{2}^{(J+1)}$ onto $C_{J}^{c}.$ First, we generalise Lemma \ref{gg} by estimating the operators $N_{0}^{(J+1)}$ and $N_{r}^{(J+1)}$:

\begin{lemma}
\label{finaal}
$$\|N_{0}^{(J+1)}(v)\|_{l^{q}L^{2}}\lesssim N^{-\frac{(q'-1)}{q'}J+\frac{(q'-1)}{100q'}(J-1)+}\|v\|_{M_{2,q}}^{2J+1},$$
and
$$\|N_{0}^{(J+1)}(v)-N_{0}^{(J+1)}(w)\|_{l^{q}L^{2}}\lesssim N^{-\frac{(q'-1)}{q'}J+\frac{(q'-1)}{100q'}(J-1)+}(\|v\|_{M_{2,q}}^{2J}+\|w\|_{M_{2,q}}^{2J})\|v-w\|_{M_{2,q}}.$$

$$\|N_{r}^{(J+1)}(v)\|_{l^{q}L^{2}}\lesssim N^{-\frac{(q'-1)}{q'}J+\frac{(q'-1)}{100q'}(J-1)+}\|v\|_{M_{2,q}}^{2J+3},$$
and
$$\|N_{r}^{(J+1)}(v)-N_{r}^{(J+1)}(w)\|_{l^{q}L^{2}}\lesssim N^{-\frac{(q'-1)}{q'}J+\frac{(q'-1)}{100q'}(J-1)+}(\|v\|_{M_{2,q}}^{2J+2}+\|w\|_{M_{2,q}}^{2J+2})\|v-w\|_{M_{2,q}}.$$
\end{lemma}
\begin{proof}
As in the proof of Lemma \ref{gg} for fixed $n^{(j)}$ and $\mu_{j}$ there are at most $o(|\mu_{j}|^{+})$ many choices for $n_{1}^{(j)},n_{2}^{(j)},n_{3}^{(j)}$. In addition, let us observe that $\mu_{j}$ is determined by $\tilde\mu_{1},\ldots,\tilde\mu_{j}$ and $|\mu_{j}|\lesssim\max(|\tilde\mu_{j-1}|,|\tilde\mu_{j}|)$, since $\mu_{j}=\tilde\mu_{j}-\tilde\mu_{j-1}.$ Then, for a fixed tree $T\in T(J)$, by Lemma \ref{indu} the estimate for the operator $\tilde q^{J,t}_{T^0,\mathbf n}$ is as follows (remember that $|\hat{\mu}_{T}|\sim|\hat{\mu}_{J}|=\prod_{k=1}^{J}|\tilde\mu_{k}|$):
$$\sum_{\substack{\mathbf n\in\mathcal R(T)\\ \mathbf n_{r}=n}}\|\tilde q^{J,t}_{T^0,\mathbf n}(\{v_{\beta}\}_{\beta\in T^\infty})\|_{2}\lesssim\sum_{\substack{\mathbf n\in\mathcal R(T)\\ \mathbf n_{r}=n}}\Big(\prod_{\beta\in T^{\infty}}\|v_{n_\beta}\|_{2}\Big)\Big(\prod_{k=1}^{J}\frac1{|\tilde\mu_{k}|}\Big),$$
and by H\"older's inequality this is bounded from above by
\begin{equation}
\label{hahah}
\Big(\sum_{\substack{|\mu_{1}|>N\\ |\tilde\mu_{j}|>(2j+1)^{3}N^{1-\frac1{100}}\\ j=2,\ldots,J}}\prod_{k=1}^{J}\frac1{|\tilde\mu_{k}|^{q'}}|\mu_{k}|^{+}\Big)^{\frac1{q'}}\Big(\sum_{\substack{\mathbf n\in\mathcal R(T)\\ \mathbf n_{r}=n}}\prod_{\beta\in T^\infty}\|v_{n_{\beta}}\|_{2}^{q}\Big)^{\frac1{q}}.
\end{equation}
The first sum behaves like $N^{-\frac{(q'-1)}{q'}J+\frac{(q'-1)}{100q'}(J-1)+}$ and for the remaining part we take the $l^{q}$ norm in $n$ and by the use of Young's inequality we are done. 

We have to make two observations for this lemma. Note that there is an extra factor $\sim J$ when we estimate the differences $N_{0}^{(J+1)}(v)-N_{0}^{(J+1)}(w)$ since $|a^{2J+1}-b^{2J+1}|\lesssim(\sum_{j=1}^{2J+1}a^{2J+1-j}b^{j-1})|a-b|$ has $O(J)$ many terms. Also, we have $c_{J}=|\mathcal I(J)|$ many summands in the operator $N_{0}^{(J+1)}$ since there are $c_{J}$ many trees of the $J$th generation and $c_{J}$ behaves like a double factorial in $J$ (see (\ref{setsetset3})). However, these observations do not cause any problem since the constant that we obtain from estimating the first sum of (\ref{hahah}) decays like a fractional power of a double factorial in $J$, or to be more precise we have 
\begin{equation}
\label{factori}
\frac{c_{J}}{\prod_{j=2}^{J}(2j+1)^{3\cdot\frac{q'-1}{q'}-}}.
\end{equation}
This fraction for large values of $J$ behaves like $J^{J}/J^{(3-\frac3{q'})J}=1/J^{(2-\frac3{q'})J}$ and in order to maintain the decay in the denominator we use the assumption of Theorem \ref{th1} namely that $1\leq q\leq 2.$ For the operator $N_{r}^{(J+1)}$ the proof is the same but in addition we use Lemma \ref{lem} for the operator $R_{2}^{t}-R_{1}^{t}$. 
\end{proof}
The estimate for the operator $N_{1}^{(J+1)}$, which generalises Lemma \ref{gg1}, is the following:

\begin{lemma}
\label{finaal2}
$$\|N_{1}^{(J+1)}(v)\|_{l^{q}L^{2}}\lesssim N^{-1+\frac2{q'}-\frac1{100q'}+(1-\frac1{100})(\frac1{q'}-1)(J-1)+}\|v\|_{M_{2,q}}^{2J+3},$$
and
$$\|N_{1}^{(J+1)}(v)-N_{1}^{(J+1)}(w)\|_{l^{q}L^{2}}\lesssim N^{-1+\frac2{q'}-\frac1{100q'}+(1-\frac1{100})(\frac1{q'}-1)(J-1)+}(\|v\|_{M_{2,q}}^{2J+2}+\|w\|_{M_{2,q}}^{2J+2})\|v-w\|_{M_{2,q}}.$$
\end{lemma}
\begin{proof}
As before, for fixed $n^{(j)}$ and $\mu_{j}$ there are at most $o(|\mu_{j}|^{+})$ many choices for $n_{1}^{(1)}, n_{2}^{(1)}, n_{3}^{(1)}$ and note that $\mu_{j}$ is determined by $\tilde\mu_{1},\ldots,\tilde\mu_{j}$. 

Let us assume that $|\tilde\mu_{J+1}|=|\tilde\mu_{J}+\mu_{J+1}|\lesssim(2J+3)^{3}|\tilde\mu_{J}|^{1-\frac1{100}}$ holds in (\ref{sesee}). Then, $|\mu_{J+1}|\lesssim|\tilde\mu_{J}|$ and for fixed $\tilde\mu_{J}$ there are at most $o(|\tilde\mu_{J}|^{1-\frac1{100}})$ many choices for $\tilde\mu_{J+1}$ and therefore, for $\mu_{J+1}=\tilde\mu_{J+1}-\tilde\mu_{J}$. For a fixed tree $T\in T(J)$ and $\alpha\in T^{\infty}$, by Lemma \ref{indu} and a trivial bound of the operator $Q^{1,t}_{n_{\alpha}}$ in $L^{2}$ (see proof of Lemma \ref{lemle}) the estimate for the operator $\tilde q^{J,t}_{T^0,\mathbf n}$ is as follows (remember that $|\hat{\mu}_{T}|\sim|\hat{\mu}_{J}|=\prod_{k=1}^{J}|\tilde\mu_{k}|$):
$$\sum_{\substack{\mathbf n\in\mathcal R(T)\\ \mathbf n_{r}=n}}\|\tilde q^{J,t}_{T^0,\mathbf n}(N_{1}^{t,\alpha}(\{v_{n_{\beta}}\}_{\beta\in T^{\infty}}))\|_{2}\lesssim$$
$$\sum_{\substack{\mathbf n\in\mathcal R(T)\\ \mathbf n_{r}=n}}\Big(\|v_{n_{\alpha_{1}}}\|_{2}\|v_{n_{\alpha_{2}}}\|_{2}\|v_{n_{\alpha_{3}}}\|_{2}\prod_{\beta\in T^{\infty}\setminus\{\alpha\}}\|v_{n_\beta}\|_{2}\Big)\Big(\prod_{k=1}^{J}\frac1{|\tilde\mu_{k}|}\Big),$$
and by H\"older's inequality we obtain the upper bound
\begin{equation}
\label{hahah2}
\Big(\sum_{\substack{|\mu_{1}|>N\\ |\tilde\mu_{j}|>(2j+1)^{3}N^{1-\frac1{100}}\\ j=2,\ldots,J}}|\tilde\mu_{J}|^{1-\frac1{100}+}\prod_{k=1}^{J}\frac1{|\tilde\mu_{k}|^{q'}}|\mu_{k}|^{+}\Big)^{\frac1{q'}}
\end{equation}
$$\Big(\sum_{\substack{\mathbf n\in\mathcal R(T)\\ \mathbf n_{r}=n}}\|v_{n_{\alpha_{1}}}\|_{2}^{q}\|v_{n_{\alpha_{2}}}\|_{2}^{q}\|v_{n_{\alpha_{3}}}\|_{2}^{q}\prod_{\beta\in T^{\infty}\setminus\{\alpha\}}\|v_{n_{\beta}}\|_{2}^{q}\Big)^{\frac1{q}}.$$
An easy calculation shows that the first sum behaves like $N^{-1+\frac2{q'}-\frac1{100q'}+(1-\frac1{100})(\frac1{q'}-1)(J-1)+}$ and then by taking the $l^q$ norm by the use of Young's inequality we are done. 

If $|\tilde\mu_{J+1}|\lesssim(2J+3)^{3}|\mu_{1}|^{1-\frac1{100}}$ holds in (\ref{sesee}), then for fixed $\mu_{j}$, $j=1,\ldots,J$, there are at most $O(|\mu_{1}|^{1-\frac1{100}})$ many choices for $\mu_{J+1}$. The same argument as above leads us to exactly the same expressions as in (\ref{hahah2}) but with the first sum replaced by the following:
$$\Big(\sum_{\substack{|\mu_{1}|>N\\ |\tilde\mu_{j}|>(2j+1)^{3}N^{1-\frac1{100}}\\ j=2,\ldots,J}}|\mu_{1}|^{1-\frac1{100}}\prod_{k=1}^{J}\frac1{|\tilde\mu_{k}|^{q'}}|\mu_{k}|^{+}\Big)^{\frac1{q'}},$$
which again is bounded from above by $N^{-1+\frac2{q'}-\frac1{100q'}+(1-\frac1{100})(\frac1{q'}-1)(J-1)+}$ and the proof is complete.
\end{proof}

\begin{remark}
\label{reme}
As it was done in \cite{GKO}, for $s>0$ we have to observe that all previous lemmata hold true if we replace the $l^{q}L^{2}$ norm by the $l^{q}_{s}L^{2}$ norm and the $M_{2,q}(\R)$ norm by the $M_{2,q}^{s}(\R)$ norm. To see this, consider $n^{(j)}$ large. Then, there exists at least one of $n_{1}^{(j)},n_{2}^{(j)},n_{3}^{(j)}$ such that $|n_{k}^{(j)}|\geq\frac13|n^{(j)}|$, $k\in\{1,2,3\}$, since we have the relation $n^{(j)}\approx n_{1}^{(j)}-n_{2}^{(j)}+n_{3}^{(j)}$. Therefore, in the estimates of the $J$th generation, there exists at least one frequency $n_{k}^{(j)}$ for some $j\in\{1,\ldots,J\}$ with the property
$$\langle n\rangle^{s}\leq 3^{js}\langle n_{k}^{(j)}\rangle^{s}\leq 3^{Js}\langle n_{k}^{(j)}\rangle ^{s}.$$
This exponential growth does not affect our calculations due to the double factorial decay in the denominator of (\ref{factori}).
\end{remark}

\subsection{Existence of Weak Solutions in the Extended Sense}

In this subsection we prove Theorem \ref{th1}. The calculations are the same as in \cite{GKO} where we just need to replace their $L^{2}$ norm by the $M_{2,q}(\R)$ norm. We will present them for the sake of completion. 

Let us start by defining the partial sum operator $\Gamma_{v_{0}}^{(J)}$ as
\begin{equation}
\label{gamaa}
\Gamma_{v_{0}}^{(J)}v(t)=v_{0}+\sum_{j=2}^{J}N_{0}^{(j)}(v)(n)-\sum_{j=2}^{J}N_{0}^{(j)}(v_{0})(n)
\end{equation}
$$+\int_{0}^{t}R_{1}^{\tau}(v)(n)+R_{2}^{\tau}(v)(n)+\sum_{j=2}^{J}N_{r}^{(j)}(v)(n)+\sum_{j=1}^{J}N_{1}^{(j)}(v)(n)\ d\tau,$$
where we have $N_{1}^{(1)}:=N_{11}^{t}$ from (\ref{main13}), $N_{0}^{(2)}:=N_{21}^{t}$ from (\ref{nex}), $N_{1}^{(2)}:=N_{31}^{t}$ from (\ref{patel}) and $N_{r}^{(2)}:=N_{4}^{t}$ from (\ref{patel2}) and $v_{0}\in M_{2,q}(\R)$ is a fixed function. 

In the following we will denote by $X_{T}=C([0,T],M_{2,q}(\R))$. Our goal is to show that the series appearing on the RHS of (\ref{gamaa}) converge absolutely in $X_{T}$ for sufficiently small $T>0$, if $v\in X_{T}$, even for $J=\infty$. Indeed, by Lemmata \ref{lem}, \ref{lemle}, \ref{finaal}, and \ref{finaal2} we obtain 
\begin{equation}
\label{argg}
\|\Gamma_{v_{0}}^{(J)}v\|_{X_{T}}\leq\|v_{0}\|_{M_{2,q}}+C\sum_{j=2}^{J}N^{-(1-\frac1{q'})(j-1)+\frac{q'-1}{100q'}(j-2)+}(\|v\|_{X_{T}}^{2j-1}+\|v_{0}\|_{M_{2,q}}^{2j-1})
\end{equation}
$$+CT\Big[\|v\|^{3}_{X_{T}}+\sum_{j=2}^{J}N^{-(1-\frac1{q'})(j-1)+\frac{q'-1}{100q'}(j-2)+}\|v\|_{X_{T}}^{2j+1}$$
$$+N^{\frac1{q'}+}\|v\|_{X_{T}}^{3}+\sum_{j=2}^{J}N^{-1+\frac2{q'}-\frac1{100q'}+(1-\frac1{100})(\frac1{q'}-1)(J-2)+}\|v\|_{X_{T}}^{2j+1}\Big].$$
Let us assume that $\|v_{0}\|_{M_{2,q}}\leq R$ and $\|v\|_{X_{T}}\leq\tilde R$, with $\tilde R\geq R\geq1.$ From (\ref{argg}) we have
\begin{equation}
\label{argg2}
\|\Gamma_{v_{0}}^{(J)}v\|_{X_{T}}\leq R+CN^{\frac1{q'}-1+}R^{3}\sum_{j=0}^{J-2}(N^{\frac1{q'}-1+\frac{q'-1}{100q'}}R^{2})^{j}+CN^{\frac1{q'}-1+}\tilde R^{3}\sum_{j=0}^{J-2}(N^{\frac1{q'}-1+\frac{q'-1}{100q'}}\tilde R^{2})^{j}
\end{equation}
$$+CT\Big[(1+N^{\frac1{q'}+})\tilde R^{3}+CN^{\frac1{q'}-1+}\tilde R^{5}\sum_{j=0}^{J-2}(N^{\frac1{q'}-1+\frac{q'-1}{100q'}}\tilde R^{2})^{j}$$
$$+N^{\frac2{q'}-1-\frac1{100q'}+}\tilde R^{5}\sum_{j=0}^{J-2}(N^{\frac1{q'}-1+\frac{q'-1}{100q'}}\tilde R^{2})^{j}\Big].$$
We choose $N=N(\tilde R)$ large enough, such that $N^{\frac1{q'}-1+\frac{q'-1}{100q'}}\tilde R^{2}=N^{99\frac{1-q'}{100q'}}\tilde R^{2}\leq\frac12$, or equivalently,
\begin{equation}
\label{argg3}
N\geq(2\tilde R^{2})^{\frac{100q'}{99(q'-1)}},
\end{equation}
so that the geometric series on the RHS of (\ref{argg2}) converge and are bounded by $2.$ Therefore, we arrive at
\begin{equation}
\label{argg4}
\|\Gamma_{v_{0}}^{(J)}v\|_{X_{T}}\leq R+2CN^{\frac1{q'}-1+}R^{3}+2CN^{\frac1{q'}-1+}\tilde R^{3}
\end{equation}
$$+CT\Big[(1+N^{\frac1{q'}+})\tilde R^{2}+2N^{\frac1{q'}-1+}\tilde R^{4}+2N^{\frac{199-100q'}{100q'}+}\tilde R^{4}\Big]\tilde R,$$
and we choose $T>0$ sufficiently small such that
\begin{equation}
\label{argg4,5}
CT\Big[(1+N^{\frac1{q'}+})\tilde R^{2}+2N^{\frac1{q'}-1+}\tilde R^{4}+2N^{\frac{199-100q'}{100q'}+}\tilde R^{4}\Big]<\frac1{10}.
\end{equation}
With the use of (\ref{argg3}) we see that $2CN^{\frac1{q'}-1+}\tilde R^{3}\leq CN^{\frac{1-q'}{100q'}+}\tilde R$ and by further imposing $N$ to be sufficiently large such that
\begin{equation}
\label{argg5}
C N^{\frac{1-q'}{100q'}+}<\frac1{10},
\end{equation}
we have
\begin{equation}
\label{argg6}
\|\Gamma_{v_{0}}^{(J)}v\|_{X_{T}}\leq R+\frac{R}{10}+\frac{\tilde R}{5}=\frac{11}{10}R+\frac15 \tilde R.
\end{equation}
Thus, for sufficiently large $N$ and sufficiently small $T>0$ the partial sum operators $\Gamma_{v_{0}}^{(J)}$ are well defined in $X_{T}$, for every $J\in\mathbb N\cup\{\infty\}$. We will write $\Gamma_{v_{0}}$ for $\Gamma_{v_{0}}^{(\infty)}$. 

Our next step is given an initial data $v_{0}\in M_{2,q}(\R)$ to construct a solution $v\in X_{T}$ in the sense of Definition \ref{def3}. To this direction, let $s>\max\{\frac1{q'},\frac12+\frac1{q}\}$ (so that $M_{2,q}^{s}(\R)$ is a Banach Algebra that embeds in $L^{2}(\R)$) and consider a sequence $\{v_{0}^{(m)}\}_{m\in\mathbb N}\in M_{2,q}^{s}(\R)\subset M_{2,q}(\R)$ whose Fourier transforms are all compactly supported (thus, all $v_{0}^{(m)}$ are smooth functions) and such that $v_{0}^{(m)}\to v_{0}$ in $M_{2,q}(\R)$ as $m\to\infty.$ Let $R=\|v_{0}\|_{M_{2,q}}+1$ and we can assume that $\|v_{0}^{(m)}\|_{M_{2,q}}\leq R$, for all $m\in\mathbb N.$ Denote by $v^{(m)}$ the local in time solution of NLS (\ref{maineq}) in $M_{2,q}^{s}(\R)$ with initial condition $v_{0}^{(m)}.$ It satisfies the Duhamel formulation:
\begin{equation}
\label{argg7}
v^{(m)}(t)=v_{0}^{(m)}+i\int_{0}^{t}N_{1}^{\tau}(v^{(m)})-R_{1}^{\tau}(v^{(m)})+R_{2}^{\tau}(v^{(m)})\ d\tau=
\end{equation}
$$v_{0}^{(m)}+\sum_{j=2}^{\infty}N_{0}^{(j)}(v^{(m)})(n)-\sum_{j=2}^{\infty}N_{0}^{(j)}(v_{0}^{(m)})(n)$$
$$+\int_{0}^{t}R_{1}^{\tau}(v^{(m)})(n)+R_{2}^{\tau}(v^{(m)})(n)+\sum_{j=2}^{\infty}N_{r}^{(j)}(v^{(m)})(n)+\sum_{j=1}^{\infty}N_{1}^{(j)}(v^{(m)})(n)\ d\tau=\Gamma_{v_{0}^{(m)}}v^{(m)}.$$
To see this it suffices to prove that the remainder term $N_{2}^{(J+1)}(v)$ given by (\ref{fina}) goes to zero in the $l^{q}L^{2}$ norm as $J$ goes to infinity for the smooth solutions $v^{(m)}.$ Indeed, we have the following lemma:

\begin{lemma}
\label{dankda}
Let $w$ be one of the smooth solutions $v^{(m)}.$ Then
$$\lim_{J\to\infty}\|N_{2}^{(J+1)}(w)\|_{l^{q}L^{2}}=0.$$
\end{lemma}
\begin{proof}
Obviously,
$$\|N_{2}^{(J+1)}(w)\|_{2}\leq \sum_{T\in T(J)}\sum_{\alpha\in T^{\infty}}\sum_{\substack{\mathbf n\in\mathcal R(T)\\ \mathbf n_{r}=n}}\|\tilde{q}^{J,t}_{T^{0}}(N_{1}^{t,\alpha}(\{w_{n_{\beta}}\}_{\beta\in T^{\infty}}))\|_{2},$$
which by Lemma \ref{indu} is bounded by 
$$\sum_{T\in T(J)}\sum_{\alpha\in T^{\infty}}\sum_{\substack{\mathbf n\in\mathcal R(T)\\ \mathbf n_{r}=n}}\prod_{\beta\in T^{\infty}\setminus\{\alpha\}}\|w_{n_{\beta}}\|_{2}\ \frac{\|N_{1}^{t}(w)(n_{\alpha})\|_{2}}{\prod_{k=1}^{J}|\tilde{\mu}_{k}|}.$$
By the definition of the operator $N_{1}^{t}(w)$ (see (\ref{main10})) and Remark \ref{expl} we arrive at the upper bound 
$$\sum_{T\in T(J)}\sum_{\alpha\in T^{\infty}}\Big[\sum_{\substack{\mathbf n\in\mathcal R(T)\\ \mathbf n_{r}=n}}\prod_{\beta\in T^{\infty}\setminus\{\alpha\}}\|w_{n_{\beta}}\|_{2}\Big(\sum_{\substack{n_{\alpha}\approx n_{1}-n_{2}+n_{3} \\ n_{1}\not\approx n_{\alpha}\not\approx n_{3}}}\|w_{n_{1}}\|_{2}\|w_{n_{2}}\|_{2}\|w_{n_{3}}\|_{2}\Big)\ \frac1{\prod_{k=1}^{J}|\tilde{\mu}_{k}|}\Big].$$
H\"older's inequality for the sum inside the brackets with indices $1/q+1/q'=1$ implies the estimate (which is basically the same as in the proof of Lemma \ref{finaal})
$$\frac1{J^{(3-\frac3{q'})J}}\sum_{T\in T(J)}\sum_{\alpha\in T^{\infty}}\Big(\sum_{\substack{\mathbf n\in\mathcal R(T)\\ \mathbf n_{r}=n}}\prod_{\beta\in T^{\infty}\setminus\{\alpha\}}\|w_{n_{\beta}}\|_{2}^{q}\Big(\sum_{\substack{n_{\alpha}\approx n_{1}-n_{2}+n_{3} \\ n_{1}\not\approx n_{\alpha}\not\approx n_{3}}}\|w_{n_{1}}\|_{2}\|w_{n_{2}}\|_{2}\|w_{n_{3}}\|_{2}\Big)^{q}\Big)^{\frac1{q}}.$$
Now we take the $l^{q}$ norm to bound $\|N_{2}^{(J+1)}(w)\|_{l^{q}L^{2}}$ by
$$\frac1{J^{(3-\frac3{q'})J}}\sum_{T\in T(J)}\sum_{\alpha\in T^{\infty}}\Big(\sum_{n\in\Z}\ \sum_{\substack{\mathbf n\in\mathcal R(T)\\ \mathbf n_{r}=n}}\prod_{\beta\in T^{\infty}\setminus\{\alpha\}}\|w_{n_{\beta}}\|_{2}^{q}(\{\|w_{n_{1}}\|_{2}\}\ast\{\|w_{n_{2}}\|_{2}\}\ast\{\|w_{n_{3}}\|_{2}\})^{q}(n_{\alpha})\Big)^{\frac1{q}},$$
and by applying Young's inequality in $l^{1}$ for $2J+1$ sequences we see that
$$\|N_{2}^{(J+1)}(w)\|_{l^{q}L^{2}}\lesssim\frac1{J^{(2-\frac3{q'})J}}\ \|w\|_{M_{2,q}}^{2J}\|\{\|w_{n_{1}}\|_{2}\}\ast\{\|w_{n_{2}}\|_{2}\}\ast\{\|w_{n_{3}}\|_{2}\}\|_{l^{q}}.$$
In general, we do not know if the $l^{q}$ norm of this convolution is finite but since $w$ is sufficiently smooth we may assume that $w\in M_{2,1}$ which is a space (actually a Banach algebra) with bigger norm than $M_{2,q}$ and we obtain
$$\|N_{2}^{(J+1)}(w)\|_{l^{q}L^{2}}\lesssim\frac1{J^{(2-\frac3{q'})J}}\ \|w\|_{M_{2,1}}^{2J+3},$$\
from which the claim follows and the proof is complete.
\end{proof}
Next we will show that (\ref{argg7}) holds in $X_{T}$ for the same time $T=T(R)>0$ independent of $m\in\mathbb N.$ Indeed, fix $m\in\mathbb N$ and observe that the norm $\|v^{(m)}\|_{X_{t}}=\|v^{(m)}\|_{C([0,t],M_{2,q})}$ is continuous in $t.$ Since $\|v_{0}^{(m)}\|_{M_{2,q}}\leq R$ there is a time $T_{1}>0$ such that $\|v^{(m)}\|_{X_{T_{1}}}\leq 4R.$ Then, by repeating the previous calculations with $\tilde R=4R$ and keeping one of the factors as $\|v^{(m)}\|_{X_{T_{1}}}$ we get
\begin{equation}
\label{argg8}
\|v^{(m)}\|_{X_{T_{1}}}=\|\Gamma_{v_{0}^{(m)}}v^{(m)}\|_{X_{T_{1}}}\leq\frac{11}{10}R+\frac15\|v^{(m)}\|_{X_{T_{1}}},
\end{equation}
if $N$ and $T_{1}$ satisfy (\ref{argg3}), (\ref{argg4,5}) and (\ref{argg5}). Therefore, we have
\begin{equation}
\label{argg9}
\|v^{(m)}\|_{X_{T_{1}}}\leq\frac{19}{10}R<2R.
\end{equation}
Thus, from the continuity of $t\to\|v^{(m)}\|_{X_{t}}$, there is $\epsilon>0$ such that $\|v^{(m)}\|X_{T_{1}+\epsilon}\leq 4R.$ Then again, from (\ref{argg8}) and (\ref{argg9}) with $T_{1}+\epsilon$ in place of $T_{1}$ we derive that $\|v^{(m)}\|_{X_{T_{1}+\epsilon}}\leq 2R$ as long as $N$ and $T_{1}+\epsilon$ satisfy (\ref{argg3}), (\ref{argg4,5}) and (\ref{argg5}). By observing that these conditions are independent of $m\in\mathbb N$ we obtain a time interval $[0,T]$ such that $\|v^{(m)}\|_{X_{T}}\leq 2R$ for all $m\in\mathbb N$. 

A similar computation on the difference, by possibly taking larger $N$ and smaller $T$ leads to the estimate
\begin{equation}
\label{argg10}
\|v^{(m_{1})}-v^{(m_{2})}\|_{X_{T}}=\|\Gamma_{v_{0}^{(m_{1})}}v^{(m_{1})}-\Gamma_{v_{0}^{(m_{2})}}v^{(m_{2})}\|_{X_{T}}\leq
\end{equation}
$$(1+\frac1{10})\|v_{0}^{(m_{1})}-v_{0}^{(m_{2})}\|_{M_{2,q}}+\frac15\|v^{(m_{1})}-v^{(m_{2})}\|_{X_{T}},$$
which implies 
\begin{equation}
\label{argg11}
\|v^{(m_{1})}-v^{(m_{2})}\|_{X_{T}}\leq c\ \|v_{0}^{(m_{1})}-v_{0}^{(m_{2})}\|_{M_{2,q}},
\end{equation}
for some $c>0$ and therefore, the sequence $\{v^{(m)}\}_{m\in\mathbb N}$ is Cauchy in the Banach space $X_{T}.$ Let us denote by $v^{\infty}$ its limit in $X_{T}$ and by $u^{\infty}=S(t)v^{\infty}$. We will show that $u^{\infty}$ satisfies NLS (\ref{maineq}) in the interval $[0,T]$ in the sense of Definition \ref{def3}. For convenience, we drop the superscript $\infty$ and write $u, v$. In addition, let $u^{(m)}:=S(t)v^{(m)}$, where $v^{(m)}$ is the smooth solution to (\ref{mainmain}) with smooth initial data $v_{0}^{(m)}$ as described above and note that $u^{(m)}$ is the smooth solution to (\ref{maineq}) with smooth initial data $u_{0}^{(m)}:=v_{0}^{(m)}$. Furthermore, $u^{(m)}\to u$ in $X_{T}$ because $v^{(m)}\to v$ in $X_{T}$ and since convergence in the modulation space $M_{2,q}(\R)$ implies convergence in the sense of distributions we conclude that $\partial_{x}u^{(m)}\to\partial_{x} u$ and $\partial_{t} u^{(m)}\to\partial_{t} u$ in $\mathcal D'((0,T)\times \R)$. Since $u^{(m)}$ satisfies NLS (\ref{maineq}) for every $m\in\mathbb N$ we have that
$$\mathcal N(u^{(m)})=u^{(m)}|u^{(m)}|^{2}=-i\partial_{t}u^{(m)}+\partial_{x}^{2}u^{(m)},$$
also converges to some distribution $w\in\mathcal S'((0,T)\times \R)$. Our claim is the following:

\begin{proposition}
\label{argg12}
Let $w$ be the limit of $\mathcal N(u^{(m)})$ in the sense of distributions as $m\to\infty$. Then, $w=\mathcal N(u)$, where $\mathcal N(u)$ is to be interpreted in the sense of Definition \ref{def2}. 
\end{proposition}
\begin{proof}
Consider a sequence of Fourier cutoff multipliers $\{T_{N}\}_{N\in\mathbb N}$ as in Definition \ref{def1}. We will prove that 
$$\lim_{N\to\infty}\mathcal N(T_{N}u)=w,$$
in the sense of distributions. Let $\phi$ be a test function and $\epsilon>0$ a fixed given number. Our goal is to find $N_{0}\in\mathbb N$ such that for all $N\geq N_{0}$ we have
\begin{equation}
\label{argg13}
|\langle w-\mathcal N(T_{N}u), \phi\rangle|<\epsilon.
\end{equation}
The LHS can be estimated as
$$|\langle w-\mathcal N(T_{N}u), \phi\rangle|\leq|\langle w-\mathcal N(u^{(m)}), \phi \rangle|+|\langle \mathcal N(u^{(m)})-\mathcal N(T_{N}u^{(m)}), \phi\rangle |$$
$$+|\langle\mathcal N(T_{N}u^{m})-\mathcal N(T_{N}u), \phi\rangle |.$$
The first term is estimated very easily since by the definition of $w$ we have that
\begin{equation}
\label{argg14}
|\langle w-\mathcal N(u^{(m)}), \phi\rangle|<\frac13\ \epsilon,
\end{equation}
for sufficiently large $m\in\mathbb N$. 

To continue, let us consider the second summand for fixed $m.$ By writing the difference $\mathcal N(u^{(m)})-\mathcal N(T_{N}u^{(m)})$ as a telescoping sum we have to estimate terms of the form
$$\Big|\int\int (I-T_{N})u^{(m)}\ |u^{(m)}|^{2}\ \phi\ dx\ dt\Big|,$$
where $I$ denotes the identity operator. By H\"older's inequality and (\ref{yeye}) we obtain that this integral is bounded by
$$\|\phi\|_{L^{2}_{T,x}}\|u^{(m)}\|_{L^{\infty}_{T,x}}^{2}\|(I-T_{N})u^{(m)}\|_{L^{2}_{T,x}}\lesssim C_{\phi}\|u^{(m)}\|_{C((0,T),M_{2,q}^{s})}^{2}\|(I-T_{N})u^{(m)}\|_{L^{2}_{T,x}}$$
$$\leq C_{\phi, m}\|(I-T_{N})u^{(m)}\|_{L^{2}_{T,x}},$$
where $L^{2}_{T,x}=L^{2}((0,T)\times\R)$. By definition of the Fourier cutoff operators, the function $\mathcal F\Big((I-T_{N})u^{(m)}(\cdot,t)\Big)(\xi)$ converges pointwise in $t$ and $\xi$ and by an application of the Dominated Convergence Theorem, there is $N_{0}=N_{0}(m)$ with the property
\begin{equation}
\label{argg15}
C_{\phi, m}\|(I-T_{N})u^{(m)}\|_{L^{2}_{T,x}}<\frac13\ \epsilon,
\end{equation}
for all $N\geq N_{0}$. 

For the last term, we need to observe two things. Firstly, let us consider the sequence $\{\mathcal N(T_{N}u^{(m)})\}_{m\in\mathbb N}$, for each fixed $N.$ By applying the iteration process that we described in the previous subsection to $\{S(-t)\mathcal N(T_{N}u^{(m)})\}_{m\in\mathbb N}$, which is basically the nonlinearity in equation (\ref{mainmain}) up to the operator $T_{N}$, we see that $\{\mathcal N(T_{N}u^{(m)})\}_{m\in\mathbb N}$ is Cauchy in $\mathcal S'((0,T)\times\R)$, as $m\to\infty$ for each fixed $N\in\mathbb N$ since the sequence $u^{(m)}$ is Cauchy in $C((0,T),M_{2,q}(\R))$. Since the multipliers $m_{N}$ of $T_{N}$ are uniformly bounded in $N$ we conclude that this convergence is uniform in $N$. 

Secondly, let us observe that for fixed $N$, $T_{N}u$ is in $C((0,T), H^{\infty}(\R))$ since $u\in M_{2,q}(\R)$ and the multiplier $m_{N}$ of $T_{N}$ is compactly supported. Hence, $\mathcal N(T_{N}u)=T_{N}u|T_{N}u|^{2}$ makes sense as a function. Therefore, for fixed $N$ by H\"older's inequality we get
$$|\langle\mathcal N(T_{N}u^{(m)})-\mathcal N(T_{N}u), \phi\rangle|\leq\|\phi\|_{L^{4}_{T,x}}(\|T_{N}u^{(m)}\|^{2}_{L^{4}_{T,x}}+\|T_{N}u\|^{2}_{L^{4}_{T,x}})\|T_{N}u^{(m)}-T_{N}u\|_{L^{4}_{T,x}}$$
$$\leq C_{\phi, \|u\|_{X_{T}}}M^{\frac34}T^{\frac34}\|u^{(m)}-u\|_{X_{T}}<\frac13\ \epsilon,$$
where the number $M=M(N)>0$ is chosen so that $\mbox{supp}(m_{N})\subset[-M,M]$. Here we used H\"older's inequality in the interval $(0,T)$ to pass from the $L^{4}$ norm to the $L^{\infty}$ norm and in the space variable an application of Parseval's identity together with the fact that the multiplier operators $T_{N}$ have compactly supported symbols $m_{N}$. Hence, $\mathcal N(T_{N}u^{(m)})$ converges to $\mathcal N(T_{N}u)$ in $\mathcal S'((0,T)\times\R)$ as $m\to\infty$ for each fixed $N$. 

From these two observations we derive that $\mathcal N(T_{N}u^{(m)})\to \mathcal N(T_{N}u)$ in $\mathcal S'((0,T)\times\R)$ as $m\to\infty$ uniformly in $N.$ Equivalently, 
\begin{equation}
\label{argg16}
|\langle\mathcal N(T_{N}u^{(m)})-\mathcal N(T_{N}u), \phi\rangle|<\frac13\ \epsilon,
\end{equation}
for all large $m$, uniformly in $N$. Therefore, (\ref{argg13}) follows by choosing $m$ sufficiently large so that (\ref{argg14}) and (\ref{argg16}) hold, and then choosing $N_{0}=N_{0}(m)$ such that (\ref{argg15}) holds. 
\end{proof}
Finally, we have shown that the function $u=u^{\infty}$ is a solution to the NLS (\ref{maineq}) in the sense of Definition \ref{def3}. 

\subsection{Unconditional Uniqueness}

In Subsections $2.1$ and $2.2$ we switched the order of space integration with time differentiation and summation in the discrete variable with time differentiation too. In the following we justify these formal computations and obtain the unconditional wellposedness of Theorem \ref{mainyeah}. 

In this subsection we assume that $u_{0}\in M_{2,q}^{s}(\R)$ with either $s\geq0$ and $1\leq q\leq\frac32$ or $\frac32<q\leq2$ and $s>\frac23-\frac1{q}$ (see also Remark \ref{bbbb} for the case $q=2$ and $s=1/6$) which by (\ref{yeye233}) and (\ref{hhh}) implies that   
\begin{equation}
\label{jo}
M_{2,q}^{s}(\R)\hookrightarrow M_{3,\frac32}(\R)\hookrightarrow L^{3}(\R).
\end{equation}
By (\ref{jo}) we know that if $u$ is a solution of NLS (\ref{maineq}) in the space $C([0,T],M_{2,q}^{s}(\R))$ then $u$ and hence $v=e^{it\partial_{x}^{2}}u$ are elements of $X_{T}\hookrightarrow C([0,T], L^{3}(\R)).$ Thus, the nonlinearity of NLS (\ref{maineq}) makes sense as an element of $C([0,T], L^{1}(\R))$ and by (\ref{main3}) we obtain that $\partial_{t}v_{n}\in C([0,T],L^{1}(\R)).$ Next, let us state a lemma:

\begin{lemma}
\label{didi}
Let $f(t,x),\partial_{t}f(t,x)\in C([0,T],L^{1}(\R^{d}))$ and define the distribution $\int_{\R^{d}}f(\cdot,x)dx$ by
$$\Big\langle \int_{\R^{d}}f(\cdot, x)dx, \phi\Big\rangle=\int_{\R}\int_{\R^{d}}f(t,x)\phi(t)dxdt,$$
with $\phi\in C^{\infty}_{c}(\R).$ Then, $\partial_{t}\int_{\R^{d}}f(\cdot,x)dx=\int_{\R^{d}}\partial_{t}f(\cdot,x)dx.$
\end{lemma}
\begin{proof}
By definition
$$\Big\langle\partial_{t}\int_{\R^{d}}f(\cdot,x)dx,\phi\Big\rangle=-\Big\langle\int_{\R^{d}}f(\cdot,x)dx,\phi'\Big\rangle=-\int_{\R}\int_{\R^{d}}f(t,x)\phi'(t)dxdt,$$
and since $f\in C([0,T],L^{1}(\R^{d}))$ we can change the order of integration by Fubini's Theorem and obtain
$$-\int_{\R^{d}}\int_{\R}f(t,x)\phi'(t)dtdx=\int_{\R^{d}}\int_{\R}\partial_{t}f(t,x)\phi(t)dtdx=\int_{\R}\int_{\R^{d}}\partial_{t}f(t,x)\phi(t)dxdt,$$
where in the first equality we used the definition of the weak derivative of $f$ and in the second equality Fubini's Theorem with the fact that $\partial_{t}f\in C([0,T],L^{1}(\R^{d})).$ The last integral is equal to 
$$\Big\langle\int_{\R^{d}}\partial_{t}f(\cdot,x)dx,\phi\Big\rangle,$$
and the proof is complete.
\end{proof}
Consider now (\ref{ttr}) for fixed $n$ and $\xi.$ We want to apply Lemma \ref{didi} to the function
$$f(t,\xi_{1},\xi_{3})=\sigma_{n}(\xi)\frac{e^{-2it(\xi-\xi_{1})(\xi-\xi_{3})}}{-2i(\xi-\xi_{1})(\xi-\xi_{3})}\ \hat{v}_{n_{1}}(\xi_{1})\hat{\bar{v}}_{n_{2}}(\xi-\xi_{1}-\xi_{3})\hat{v}_{n_{3}}(\xi_{3}),$$
where $\xi\approx n, \xi_{1}\approx n_{1},\xi_{3}\approx n_{3}, \xi-\xi_{1}-\xi_{3}\approx -n_{2}$ and $(n,n_{1},n_{2},n_{3})\in A_{N}(n)^{c}$ given by (\ref{idid}). Notice that $f, \partial_{t}f \in C([0,T],L^{1}(\R^{2}))$ since $v\in C([0,T],M_{2,q}^{s}(\R))$ and $\partial_{t}v_{n}\in C([0,T],L^{1}(\R))$ for all integers $n$. Thus,
$$\partial_{t}\Big[\int_{\R^2}\sigma_{n}(\xi)\frac{e^{-2it(\xi-\xi_{1})(\xi-\xi_{3})}}{-2i(\xi-\xi_{1})(\xi-\xi_{3})}\ \hat{v}_{n_{1}}(\xi_{1})\hat{\bar{v}}_{n_{2}}(\xi-\xi_{1}-\xi_{3})\hat{v}_{n_{3}}(\xi_{3})d\xi_{1}d\xi_{3}\Big]=$$
$$\int_{\R^2}\sigma_{n}(\xi)\partial_{t}\Big[\sigma_{n}(\xi)\frac{e^{-2it(\xi-\xi_{1})(\xi-\xi_{3})}}{-2i(\xi-\xi_{1})(\xi-\xi_{3})}\ \hat{v}_{n_{1}}(\xi_{1})\hat{\bar{v}}_{n_{2}}(\xi-\xi_{1}-\xi_{3})\hat{v}_{n_{3}}(\xi_{3})\Big]d\xi_{1}d\xi_{3}=$$
$$\int_{\R^2}\sigma_{n}(\xi)\partial_{t}\Big[\frac{e^{-2it(\xi-\xi_{1})(\xi-\xi_{3})}}{-2i(\xi-\xi_{1})(\xi-\xi_{3})}\Big]\hat{v}_{n_{1}}(\xi_{1})\hat{\bar{v}}_{n_{2}}(\xi-\xi_{1}-\xi_{3})\hat{v}_{n_{3}}(\xi_{3})d\xi_{1}d\xi_{3}+$$
$$\int_{\R^2}\sigma_{n}(\xi)\frac{e^{-2it(\xi-\xi_{1})(\xi-\xi_{3})}}{-2i(\xi-\xi_{1})(\xi-\xi_{3})}\partial_{t}\Big[\hat{v}_{n_{1}}(\xi_{1})\hat{\bar{v}}_{n_{2}}(\xi-\xi_{1}-\xi_{3})\hat{v}_{n_{3}}(\xi_{3})\Big]d\xi_{1}d\xi_{3}.$$
In the second equality we used the product rule which is applicable since $v\in C([0,T],L^{3}(\R))$ implies that $\partial_{t}v_{n}\in C([0,T],L^{1}(\R)).$

Finally it remains to justify the interchange of differentiation in time and summation in the discrete variable but this is done in exactly the same way as in \cite{GKO} (Lemma $5.1$). Similar arguments justify the interchange on the $J$th step of the infinite iteration procedure. 

Thus, for $v\in C([0,T],M_{2,q}^{s}(\R))$ with $M_{2,q}^{s}(\R)\hookrightarrow L^{3}(\R)$ we can repeat the calculations of Subsections $2.1$ and $2.2$ and $2.3$ to obtain the following expression in $X_{T}$ for the solution $u$ of NLS (\ref{maineq}) with initial data $u_{0}$
\begin{equation} 
\label{antt1}
u=\Gamma_{u_{0}}u+\lim_{J\to\infty}\int_{0}^{t}N_{2}^{(J+1)}(u)d\tau,
\end{equation}
where the limit is an element of $X_{T}.$ Its existence follows from the fact that the operators $\Gamma_{u_{0}}^{(J)}u$ converge to $\Gamma_{u_{0}}u$ in the norm of $X_{T}$ as $J\to\infty$. The important estimate about the remainder operator $N_{2}^{(J)}$ is the following:

\begin{lemma}
\label{finafinafina}
$$\lim_{J\to\infty}\|N_{2}^{(J)}(v)\|_{l^{\infty}L^{2}}=0.$$
\end{lemma} 
\begin{proof}
By (\ref{fina1}) we can write the remainder operator as the following sum
\begin{equation}
\label{ppp}
N_{2}^{(J)}(v)(n)=\partial_{t}(N_{0}^{(J+1)}(v)(n))+\sum_{T\in T(J)}\sum_{\alpha\in T^{\infty}}\sum_{\substack{\mathbf n\in\mathcal R(T)\\ \mathbf n_{r}=n}}\tilde q^{J,t}_{T^0,\mathbf n}(\partial_{t}^{(\alpha)}(\{v_{n_{\beta}}\}_{\beta\in T^{\infty}})),
\end{equation}
where we define the action of $\partial_{t}^{(\alpha)}$ onto the set of functions $\{v_{n_{\beta}}\}_{\beta\in T^{\infty}}$ to be the same set of functions except for the $\alpha$ node where we replace $v_{n_{\alpha}}$ by the function $\partial_{t}v_{n_{\alpha}}.$ 

We control the first summand $\partial_{t}(N_{0}^{(J+1)}(v)(n))$ by Lemma \ref{finaal}. For the last summand of the RHS of (\ref{ppp}) we estimate its $L^{2}$ norm in exactly the same way as in the proof of Lemma (\ref{dankda}) and arrive at the upper bound
$$\sum_{T\in T(J)}\sum_{\alpha\in T^{\infty}}\sum_{\substack{\mathbf n\in\mathcal R(T)\\ \mathbf n_{r}=n}}\prod_{\beta\in T^{\infty}\setminus\{\alpha\}}\|v_{n_{\beta}}\|_{2}\ \frac{\|\partial_{t}v_{n_{\alpha}}\|_{2}}{\prod_{k=1}^{J}|\tilde{\mu}_{k}|},$$
which by H\"older's inequality with exponents $\frac1{q}+\frac1{q'}=1$ implies
$$\frac1{J^{(3-\frac3{q'})J}}\sum_{T\in T(J)}\sum_{\alpha\in T^{\infty}}\Big(\sum_{\substack{\mathbf n\in\mathcal R(T)\\ \mathbf n_{r}=n}}\prod_{\beta\in T^{\infty}\setminus\{\alpha\}}\|v_{n_{\beta}}\|_{2}^{q}\|\partial_{t}v_{n_{\alpha}}\|_{2}^{q}\Big)^{\frac1{q}}.$$
Then for the sum inside the parenthesis we apply Young's inequality in the discrete variable where for the first $2J$ functions we take the $l^{1}$ norm and for the last the $l^{\infty}$ norm we arrive at the estimate 
$$\|v\|_{M_{2,q}}^{2J}\sup_{n\in\Z}\|\partial_{t}v_{n}\|_{2}=\|v\|_{M_{2,q}}^{2J}\|\partial_{t}v_{n}\|_{l^{\infty}L^{2}}.$$
Since by (\ref{main3}) we have $\partial_{t}v_{n}=e^{it\partial_{x}^{2}}\Box_{n}(|u|^{2}u)$ it is straightforward to obtain
$$\|\partial_{t}v_{n}\|_{l^{\infty}L^{2}}\lesssim\|v\|_{M_{2,q}}^{3}.$$
Indeed, from (\ref{Sch}) and since $\Box_{n}(|u|^{2}u)$ is nicely localised it suffices to estimate 
$$\|\Box_{n}(|u|^{2}u)\|_{2}\lesssim\|\Box_{n}(|u|^{2}u)\|_{1}\lesssim\||u|^{2}u\|_{1}=\|u\|_{3}^{3}\lesssim\|u\|_{M_{2,q}}^{3}=\|v\|_{M_{2,q}}^{3},$$ 
where we used (\ref{Bern1}) and (\ref{Bern}). Therefore, putting everything together we have
$$\|N_{2}^{(J)}(v)\|_{l^{\infty}L^{2}}\lesssim\frac1{J^{(2-\frac3{q'})J}}\|v\|_{M_{2,q}}^{2J+3},$$
which finishes the proof. 
\end{proof}
This lemma implies that $\lim_{J\to\infty}\int_{0}^{t}N_{2}^{(J+1)}(u)d\tau$ is equal to $0$ in $X_{T}.$ From this we obtain the unconditional uniqueness of NLS (\ref{maineq}) since if there are two solutions $u_{1}$ and $u_{2}$ with the same initial datum $u_{0}$ we obtain by (\ref{argg11})
$$\|u_{1}-u_{2}\|_{X_{T}}=\|\Gamma_{u_{0}}u_{1}-\Gamma_{u_{0}}u_{2}\|_{X_{T}}\lesssim\|u_{0}-u_{0}\|_{M_{2,q}^{s}}=0.$$

\textbf{Acknowledgments}: The author gratefully acknowledge financial support by the Deu\-tsche Forschungs\-gemeinschaft (DFG) through CRC 1173. He would also like to thank Peer Kunstmann from KIT for his helpful comments and fruitful discussions. Finally, he would like to thank the referees of the paper for their constructive criticism.

\end{section}

\end{document}